\documentclass{article}
\pdfoutput=1
\usepackage{amssymb,amsmath,amscd}
\usepackage{amsthm}
\usepackage{latexsym}
\usepackage{color}
\newtheorem{theorem}{Theorem}[section]
\newtheorem{lemma}[theorem]{Lemma}
\newtheorem{corollary}[theorem]{Corollary}
\newtheorem{proposition}[theorem]{Proposition}
\theoremstyle{definition}

\newtheorem{definition}[theorem]{Definition}
\theoremstyle{remark}
\newtheorem{ownremark}[theorem]{Remark} 
\numberwithin{equation}{section}
\newcommand{\ignore}[1]{}

\newcommand{\xup}{\ensuremath{\mathcal X}_{u,p} }
\newcommand{\xjup}{\ensuremath{\mathcal X}^{(j)}_{u,p} }

\newcommand{\beq}{\begin{eqnarray}}
\newcommand{\eeq}{\end{eqnarray}}
\newcommand{\beqq}{\begin{eqnarray*}}
\newcommand{\eeqq}{\end{eqnarray*}}
\newcommand{\hra}{\hookrightarrow}
\newcommand{\bit}{\begin{itemize}}
\newcommand{\eit}{\end{itemize}}
\newcommand{\bli}{\begin{list}{}{\labelwidth6mm\leftmargin8mm}}
\newcommand{\eli}{\end{list}}
\newcommand{\bpr}{\begin{proof}}
\newcommand{\epr}{\end{proof}}
\newcommand{\ds}{\displaystyle}
\newcommand{\R}{{\mathbb R}}
\newcommand{\real}{\R}
\newcommand{\Rn}{{\mathbb R}^{d}} 
\newcommand{\N}{\ensuremath{\mathbb N}}
\newcommand{\No}{\N_{0}}
\newcommand{\nat}{\N}
\newcommand{\no}{\No}
\newcommand{\Z}{\mathbb Z} 
\newcommand{\Zn}{\Z^{d}}
\newcommand{\C}{{\mathbb C}}
\newcommand{\ve}{\ensuremath\varepsilon}

\newcommand{\ls}{\lesssim}

\newcommand{\dint}{\;\mathrm{d}}
\newcommand{\id}{\ensuremath{\operatorname{id}}}
\newcommand{\whole}[1]{\ensuremath\left\lfloor #1 \right\rfloor}
\newcommand{\M}{\ensuremath{{\cal M}_{u,p}}}  
\newcommand{\Me}{\ensuremath{{\cal M}_{u_1,p_1}}}  
\newcommand{\Mz}{\ensuremath{{\cal M}_{u_2,p_2}}}  
\newcommand{\MB}{\ensuremath{{\cal N}^{s}_{u,p,q}}}  
\newcommand{\bt}{{B}_{p,q}^{s,\tau}}
\newcommand{\ft}{{F}_{p,q}^{s,\tau}}

\newcommand{\MF}{\ensuremath{{\cal E}^{s}_{u,p,q}}}  
\newcommand{\mb}{\ensuremath{n^{s}_{u,p,q}}}  
\newcommand{\mbe}{\ensuremath{n^{s_1}_{u_1,p_1,q_1}}}  
\newcommand{\mbz}{\ensuremath{n^{s_2}_{u_2,p_2,q_2}}}  
\newcommand{\mmb}{\ensuremath{m^{2^{jd}}_{u,p}}} 
\newcommand{\mmbet}{\ensuremath{{m}^{2^{jd}}_{u_1,p_1}}}  
\newcommand{\mmbzt}{\ensuremath{{m}^{2^{jd}}_{u_2,p_2}}}  

\newcommand{\ms}{\ensuremath{{ m}_{u,p}}}
\newcommand{\mse}{\ensuremath{{ m}_{u_1,p_1}}}
\newcommand{\msz}{\ensuremath{{ m}_{u_2,p_2}}}

\title{Morrey Sequence Spaces: Pitt's Theorem and compact embeddings}
\author{Dorothee D. Haroske, \\  
\emph{Institute of Mathematics, University of Rostock, 18057 Rostock, Germany} \\ 
Leszek Skrzypczak\thanks{The author was  supported by  National Science Center, Poland, Grant No. 2014/15/B/ST1/00164.}\\
\emph{Faculty of Mathematics and Computer Science, Adam  Mickiewicz University,}\\
\emph{ ul. Umultowska  87, 61-614 Pozna\'n, Poland}}
\date{}

\begin{document}
\maketitle

\begin{abstract}
Morrey (function) spaces and, in particular, smoothness spaces of Besov-Morrey or Triebel-Lizorkin-Morrey type enjoyed a lot of interest recently. Here we turn our attention to Morrey sequence spaces $m_{u,p}=m_{u,p}(\mathbb{Z}^d)$, $0<p\leq u<\infty$, which have yet been considered almost nowhere. They are defined as natural generalisations of the classical $\ell_p$ spaces. 
We consider some basic features, embedding properties, the pre-dual, a corresponding version of Pitt's compactness theorem, and can further  characterise the compactness of embeddings of related finite-dimensional spaces.
\end{abstract}

%






\section*{Introduction}
Morrey (function) spaces and, in particular, smoothness spaces of Besov-Morrey or Triebel-Lizorkin-Morrey type were studied in recent years quite intensively and systematically. Decomposition methods like atomic or wavelet characterisations require suitably adapted sequence spaces. This has been done to some extent already. We are interested in related {\em sequence spaces of Morrey type}, but first we briefly review some basic facts about the much more prominent {\em function spaces of Morrey type}. 

Originally, Morrey spaces were introduced in \cite{Mor38}, when studying solutions of second order quasi-linear elliptic equations in the framework of Lebesgue spaces.  They can be understood as a complement (generalisation) of the Lebesgue spaces $L_p(\Rn)$. In particular, the Morrey space $\M$, $0<p\leq u<\infty$, is defined as the collection of all  complex-valued Lebesgue measurable functions on $\Rn$ such that 
\begin{equation}\label{morrey-func}
\|f|\mathcal{M}_{u,p}(\Rn)\|=\sup_{x\in \Rn, R>0} R^{d(\frac1u-\frac1p)} \left(\int_{B(x,R)}|f(y)|^p\dint y\right)^{\frac 1p}<\infty,
\end{equation}
where $B(x,R)=\{y\in\Rn: |x-y|<R\}$ are the usual balls centred at $x\in\Rn$ with radius $R>0$. Obviously, $\mathcal{M}_{p,p}(\Rn)=L_p(\Rn)$, and $\M(\Rn)=\{0\}$ if $p>u$. Moreover, $\mathcal{M}_{\infty,p}(\Rn)=L_\infty(\Rn)$ such that the usual assumption is $p\le u<\infty$. As can be seen from the definition, Morrey spaces describe the local behaviour of the $L_p$ norm, which makes them useful when describing the local behaviour of solutions of nonlinear partial differential equations, cf.  \cite{KY94,LeR07,LeR12,LeR13,Maz03,Maz03b,Tay92}.  Furthermore,  applications in  harmonic analysis and potential analysis can be found in the papers  \cite{AX04,AX11,AX12}. For more information we refer to the books \cite{Ad15} and \cite{SYY10} and, in particular, to the fine surveys \cite{s011,s011a} by Sickel.

As for the smoothness spaces of Morrey type, aside of  Besov-Morrey spaces $\MB(\Rn)$ in \cite{KY94,Maz03,Maz03b}, and their counterparts Triebel-Lizorkin-Morrey spaces $\MF(\Rn)$, cf. \cite{TX}, their atomic and wavelet characterisations were already described in the papers \cite{Saw2,Saw1,ST2,ST1,MR-1}, and we simplified the appearing sequence spaces $\mb$ in  \cite{HaSk-bm1} to some extent. There are further related approaches to  Besov-type spaces $\bt(\Rn)$ and Triebel-Lizorkin-type spaces $\ft(\Rn)$, cf. \cite{SYY10} with forerunners in \cite{ElBaraka1,ElBaraka2, ElBaraka3,yy1,yy2}. Triebel provided a third approach, so-called local and hybrid spaces, in \cite{t13,t14}, but they coincide with appropriately chosen spaces of type $\bt(\Rn)$ or $\ft(\Rn)$, cf. \cite{ysy2}.

Recently, based on some discussion at the conference  
`Banach Spaces and Operator Theory with Applications' in Pozna\'n in July 2017 we found that Morrey sequence spaces $m_{u,p}=m_{u,p}(\mathbb{Z}^d)$, $0<p\leq u<\infty$, have been considered almost nowhere. As to the best of our knowledge there is only the paper \cite{GuKiSch} so far which is devoted to this subject. They are defined as natural generalisations of $\ell_p=\ell_p(\mathbb{Z}^d)$ via 
\begin{align*}
m_{u,p} = \bigg\{\lambda=\{\lambda_k\}_{k\in \mathbb{Z}^d} &\subset \mathbb{C}: 
 \\
 \|\lambda| m_{u,p} \| = & \sup_{j\in\mathbb{N}_0; m\in \mathbb{Z}^d} |Q_{-j,m}|^{\frac{1}{u}-\frac{1}{p}} \Big(\sum_{k:\, Q_{0,k}\subset Q_{-j,m}} \!\!\!|\lambda_k|^p\Big)^\frac{1}{p} < \infty \bigg\},
\end{align*}
where $Q_{i,m}$ are dyadic cubes of side length $2^{-i}$, $i\in\mathbb{Z}$, $m\in\mathbb{Z}^d$. Clearly, $m_{p,p}= \ell_p$.

We consider some basic features in Section~\ref{sec-1} and present our main embedding result, Theorem~\ref{cont}, in Section~\ref{sec-2}, which reads as follows: Let    $0<p_1\leq u_1<\infty$ and  $0<p_2\leq u_2<\infty$.  Then the embedding 
\[
	 \mse\hookrightarrow \msz
\]
is continuous  if, and only if, the following conditions hold
\[
u_1 \le u_2\qquad \text{and}\qquad \frac{p_2}{u_2} \le\frac{p_1}{u_1}.
\]
The embedding is never compact.

In Section~\ref{sec-3} we can describe the pre-dual $\xup$ of $\ms$, $1\leq p<u<\infty$, which is a separable Banach space, unlike $\ms$.

Dealing with   the closure  $\ms^{00}$ of finite sequences in $\ms$, we obtain a counterpart to Pitt's theorem \cite{Pitt} in our setting as follows, see Theorem~\ref{pitt-ms} below: Let $1<p<u<\infty$ and $1\le q<\infty$. Then any bounded linear operator
\[T: \ms^{00} \to \ell_q \]
is compact. The above sequence spaces are not rearrangement invariant. Further  information about the Pitt theorem  in rearrangement invariant setting  can be found in \cite{DMR} and \cite{LM}.    

Finally, we can further  characterise the compactness of embeddings of related finite-dimensional spaces and receive for the asymptotic behaviour of the dyadic entropy numbers of such a finite-dimensional embedding that
\[
e_k(\id_j:\mmbet \rightarrow \mmbzt) \sim  2^{-k2^{-jd}} \ 2^{jd\left(\frac{1}{u_2}-\frac{1}{u_1}\right)},
\]
where $j\in\nat$, $0<p_i\leq u_i<\infty$, $i=1,2$, and $k\in \no$ with $k\gtrsim 2^{jd}$.

\section{Morrey sequence spaces}\label{sec-1}
\subsection{Preliminaries}
First we fix some notation. By $\nat$ we denote the \emph{set of natural numbers},
by $\no$ the set $\nat \cup \{0\}$,  and by $\Zn$ the \emph{set of all lattice points
in $\Rn$ having integer components}.
For $a\in\real$, let   $\whole{a}:=\max\{k\in\Z: k\leq a\}$ and $a_+:=\max\{a,0\}$.
All unimportant positive constants will be denoted by $C$, occasionally with
subscripts. By the notation $A \lesssim B$, we mean that there exists a positive constant $C$ such that
 $A \le C \,B$, whereas  the symbol $A \sim B$ stands for $A \ls B \ls A$.
We denote by $|\cdot|$ the Lebesgue measure when applied to measurable subsets of $\Rn$.

Given two (quasi-)Banach spaces $X$ and $Y$, we write $X\hookrightarrow Y$
if $X\subset Y$ and the natural embedding of $X$ into $Y$ is continuous.

For $0<p<\infty$ we denote by $\ell_p=\ell_p(\Zn)$, 
\begin{align*} 
\ell_p(\Zn) = \Big\{\lambda=  \{\lambda_k\}_{k\in \Zn}\subset \C:           \|\lambda| \ell_p \| = \Big(\sum_{k\in\Zn} |\lambda_k|^p\Big)^\frac{1}{p} < \infty \Big\}, 
\end{align*} 
complemented by
\begin{align*} 
\ell_\infty(\Zn) = \Big\{\lambda=  \{\lambda_k\}_{k\in \Zn}\subset \C:       \|\lambda| \ell_\infty \| = \sup_{k\in\Zn} |\lambda_k| < \infty \Big\}. 
\end{align*}
If $\{\lambda^*_\nu\}_{\nu\in\nat}$ stands for a non-increasing rearrangement of a sequence $\lambda=\{\lambda_k\}_{k\in\Zn}\in\ell_u(\Zn)$, $0<u<\infty$, then 
\begin{align*} 
\ell_{u,\infty}(\Zn) = \Big\{\lambda=  \{\lambda_k\}_{k\in \Zn}\subset \C:       \|\lambda| \ell_{u,\infty} \| = \sup_{\nu\in\nat} \nu^{1/u} \lambda^*_\nu < \infty \Big\} 
\end{align*}
denote the Lorentz sequence spaces, as usual. Finally, we adopt the custom to denote by $c=c(\Zn)$, $c_0=c_0(\Zn)$, and $c_{00}=c_{00}(\Zn)$ the corresponding subspaces of $\ell_\infty(\Zn)$ of convergent, null and finite sequences, respectively, that is
\begin{align*}
c & = \{ \lambda= \{\lambda_k\}_{k\in \Zn}\in\ell_\infty: \ \exists\ \mu\in\C: |\lambda_k-\mu|\xrightarrow[|k|\to\infty]{} 0\Big\},\\
c_0 & = \{ \lambda= \{\lambda_k\}_{k\in \Zn}\in\ell_\infty: \ |\lambda_k|\xrightarrow[|k|\to\infty]{} 0\Big\},\\
c_{00} & = \{ \lambda= \{\lambda_k\}_{k\in \Zn}\in\ell_\infty: \ \exists\ r_0\in\no: \lambda_k=0\quad\text{for}\quad |k|>r_0\Big\}.
\end{align*}
As we mostly deal with sequence spaces on $\Zn$ we shall often omit it from their notation, for convenience.

\subsection{The concept}
Let $Q_{j,m}$,  $j\in \Z $, $m\in \Zn$, denote the usual dyadic cubes in $\Rn$ i.e. $Q_{0,0}=[0,1)^d$ and $Q_{j,m}= 2^{-j}(m+Q_{0,0})$.  

\begin{definition} \label{D-ms}
Let $0<p\le u < \infty$. We define $\ms=\ms(\Zn)$ by
\begin{align*} 
\ms(\Zn) = \Big\{\lambda= & \{\lambda_k\}_{k\in \Zn}\subset \C:  \\
             &              \|\lambda| \ms \| = \sup_{j\in\no; m\in \Zn} |Q_{-j,m}|^{\frac{1}{u}-\frac{1}{p}} \Big(\sum_{k:\, Q_{0,k}\subset Q_{-j,m}} \!\!\!|\lambda_k|^p\Big)^\frac{1}{p} < \infty \Big\}. 
\end{align*} 
\end{definition}

\begin{ownremark}
Obviously, Definition~\ref{D-ms} gives the discrete counterpart of $\M(\Rn)$ in view of \eqref{morrey-func}. In \cite{GuKiSch} the corresponding one-dimensional counterpart was introduced and studied.
\end{ownremark}

\begin{proposition}\label{Prop-1-ms}
Let $0<p\le u < \infty$. 
\bli
\item[{\bfseries\upshape (i)}]
$\ms$ is a (quasi-) Banach space. 
\item[{\bfseries\upshape (ii)}]
If $u=p$, then $m_{u,u} = \ell_u$, if $u<p$, then $\ms=\{0\}. $ If $p_1\leq p_2\leq u$, then $m_{u,p_2} \hookrightarrow m_{u,p_1}$.
 \item[{\bfseries\upshape (iii)}]
For any $p$ and $u$  we have $\ms\hookrightarrow \ell_\infty$.
\item[{\bfseries\upshape (iv)}]
If $p<u$, then $\ell_{u,\infty}\hookrightarrow \ms$.
\item[{\bfseries\upshape (v)}]
If $p<u$, then $\ms$ and $c_0$, $c$ are incomparable, that is, $\ms\not\subset c_0$, $\ms\not\subset c$, $c_0\not\subset \ms$, $c\not\subset \ms$.
\eli
\end{proposition}

\begin{proof}
Part (i) is standard, the completeness can be shown similar to the (one-dimensional) counterpart in \cite{GuKiSch}.\\
The first two assertions in (ii) are obvious, the monotonicity in $p$ is a matter of H\"older's inequality.
Concerning (iii), clearly for any $m\in\Zn$,
\begin{align*}
|\lambda_m|= & |Q_{0,m}|^{\frac1u-\frac1p} \left(\sum_{k:Q_{0,k}\subset Q_{0,m}} \!\!\!|\lambda_k|^p\right)^{\frac1p}\\
\leq & \sup_{j\in\no; m\in \Zn} |Q_{-j,m}|^{\frac{1}{u}-\frac{1}{p}} \Big(\sum_{k:\, Q_{0,k}\subset Q_{-j,m}}\!\!\! |\lambda_k|^p\Big)^\frac{1}{p}
=  \|\lambda|\ms\|,
\end{align*}
such that finally, taking the supremum over all $m\in\Zn$,
\[\|\lambda|\ell_\infty\| \leq \|\lambda|\ms\| .\]

We prove (iv). 
 Let $\{\lambda^*_\nu\}_{\nu\in\nat}$ be a non-increasing rearrangement of a sequence $\lambda=\{\lambda_k\}_{k\in\Zn}$. Then for any cube $Q_{-j,m}$ we have  
  \begin{align*}
\Big( \sum_{k:Q_{0,k}\subset Q_{-j,m}}|\lambda_k|^p\Big)^{1/p} &\le \Big( \sum_{\nu=1}^{2^{jd}}|\lambda^*_\nu|^p\Big)^{1/p}
\le \sup_{r\in\nat} r^{1/u} \lambda^*_r \, \Big( \sum_{\nu=1}^{2^{jd}}\nu^{-p/u}\Big)^{1/p}\\
&\le C  \|\lambda|\ell_{u,\infty}\|\, |Q_{-j,m}|^{\frac{1}{p}- \frac{1}{u}} ,
\end{align*}
since $p<u$ and $|Q_{-j,m}|=2^{jd}$. Hence $\|\lambda|\ms\|\leq C\|\lambda|\ell_{u,\infty}\|$.\\
It remains to deal with (v). Consider first the constant sequence $\lambda^1=\{1\}_{k\in\Zn}\in c$. Then 
\[
\|\lambda^1|\ms\| = 
\sup_{j\in\no; m\in \Zn} |Q_{-j,m}|^{\frac{1}{u}-\frac{1}{p}} |Q_{-j,m}|^{\frac1p} = 
\sup_{j\in\no} 2^{jd/u} = \infty,
\]
which disproves $c\subset \ms$ (and simultaneously strengthens (iii) by $\ms\subsetneq \ell_\infty$). A slight modification disproves $c_0\subset\ms$: choose $\ve$ such that $0<\ve<\frac{d}{u}$, and 
consider $\tilde{\lambda}= \{\tilde{\lambda}_k\}_{k\in\Zn}$ given by $\tilde{\lambda}_k=|k|^{-\ve}$. Then $\tilde{\lambda}\in c_0$. On the other hand,
\[
\|\tilde{\lambda}|\ms\| \geq c\ 
\sup_{j\in\no} |Q_{-j,0}|^{\frac{1}{u}-\frac{1}{p}} 2^{-j\ve} |Q_{-j,0}|^{\frac1p} = 
c \sup_{j\in\no} 2^{j(d/u-\ve)} = \infty,
\]
which gives $c_0\not\subset\ms$. \\
Now consider a special sequence $\lambda = \{\lambda_k\}_{k\in\Zn}$ which looks as follows,
\[
\lambda_k = \begin{cases} 1, & \text{if}\ k=(2^r, 0, \dots, 0)\ \text{for some}\  r\in\nat, \\ 0, & \text{otherwise}.
\end{cases}
\]
Obviously $\lambda\not\in c$, in particular, $\lambda\not\in c_0$. Now, by construction,
\[
\|\lambda|\ms\|\leq c\ \sup_{j\in\no} |Q_{-j,0}|^{\frac1u-\frac1p} \Big(\sum_{k:Q_{0,k}\subset Q_{-j,0}} \!\!\!|\lambda_k|^p\Big)^{\frac1p} \leq c'\ \sup_{r\in\nat} 2^{rd(\frac1u-\frac1p)} r^{\frac1p} < \infty.
\]
So the subspaces $c_0$, $c$ and $\ms$ of $\ell_\infty$ are incomparable in the above sense.
\end{proof}

\begin{ownremark}
As in case of the function spaces $\M(\Rn)$ one might complement the Definition~\ref{D-ms} in case of $0<p\leq u=\infty$ by
\begin{align*} 
m_{\infty,p}(\Zn) = \Big\{\lambda=&\{\lambda_k\}_{k\in \Zn}\subset \C: \\              & \|\lambda| m_{\infty,p} \| = \sup_{j\in\no; m\in \Zn} |Q_{-j,m}|^{-\frac{1}{p}} \Big(\sum_{k: Q_{0,k}\subset Q_{-j,m}} \!\!\!|\lambda_k|^p\Big)^\frac{1}{p} < \infty \Big\}, 
\end{align*} 
where in case of $p=u=\infty$ the latter sum has to be replaced by the supremum, as usual. Using this definition we can show that $m_{\infty, p} = \ell_\infty$. This is obvious for $p=u=\infty$, so let us assume $p<u=\infty$. 
In view of Proposition~\ref{Prop-1-ms}(iii) it remains to show that $\ell_\infty\hookrightarrow m_{\infty,p}$ if $p<u=\infty$. Let $\lambda \in \ell_\infty$.
Thus for any $j\in\no$ and $m\in\Zn$,
\[
|Q_{-j,m}|^{-\frac{1}{p}}
\Big(\sum_{k:Q_{0,k}\subset Q_{-j,m}}\!\!\! |\lambda_k|^p\Big)^\frac{1}{p} \leq \
|Q_{-j,m}|^{-\frac{1}{p}} \|\lambda|\ell_\infty\|\ |Q_{-j,m}|^{\frac1p} = \|\lambda|\ell_\infty\|,
\]
which results in $\|\lambda|m_{\infty,p}\|\leq \|\lambda|\ell_\infty\|$.
\end{ownremark}

\begin{ownremark}\label{rem-u=infty}
Let $Q$ denote an arbitrary closed cube in $\R^d$ with $|Q|\ge 1$. We put 
\begin{align}
\|\lambda | \ms \|_{(1)} = & \sup_{Q: |Q|\ge 1} |Q|^{\frac{1}{u}-\frac{1}{p}} \Big(\sum_{k: Q_{0,k}\subset Q} \!\!|\lambda_k|^p\Big)^\frac{1}{p}, \\ 
\| \lambda | \ms \|_{(2)} = & \sup_{Q: |Q|\ge 1} |Q|^{\frac{1}{u}-\frac{1}{p}} \Big(\sum_{k: Q_{0,k}\cap Q \not= \emptyset} \!\!|\lambda_k|^p\Big)^\frac{1}{p} .
\end{align}   
Then $\|\lambda | \ms \|_{(1)}$ and $\|\lambda | \ms \|_{(2)}$ are equivalent norms in $\ms$. It is obvious that  $\|\lambda | \ms \| \le \|\lambda | \ms \|_{(1)} \le \|\lambda | \ms \|_{(2)}$.  On the other hand let  $Q$ be a cube centred at $x_0$ with size $r\ge 1$. We take  a cube $\widetilde{Q}$ centred at $x_0$ with size $r+2$. If we choose $j$ in such a way that $2^{j-1}< r+2 \le 2^{j}$, then there are at most $2^d$  dyadic cubes $Q_{-j,m_i}$ such   that $\overline{Q}_{-j,m_i}$ cover $\widetilde{Q}$.  Then $|Q_{-j,m}|\le 2^d |\widetilde{Q}| \le 2^{2d} |Q|$ and in consequence. 
\[ \|\lambda | \ms \|_{(2)}\le  c_{d,u,p}\|\lambda | \ms \|_{(1)} \le C_{d,u,p}\|\lambda | \ms \| . \]      
\end{ownremark}

\begin{proposition}
Let  $N\in\nat$, $0<p_j\leq u_j<\infty$, $j=1, \dots, N$, and $\lambda^{(j)}=\{\lambda_k^{(j)}\}_{k\in\Zn}\in m_{u_j, p_j}$, $j=1, \dots, N$. 
Then $\lambda^{(1)}\cdots \lambda^{(N)}=\{\lambda_k^{(1)}\cdots \lambda^{(N)}_k\}_{k\in\Zn}\in\ms$, with
\[
\| \lambda^{(1)}\cdots \lambda^{(N)}| \ms\| \leq \|\lambda^{(1)}|\mse\| \cdots \|\lambda^{(N)}|m_{u_N, p_N}\|,
\]
\text{where}
\[ \frac1u= \sum_{j=1}^N \frac{1}{u_j}\quad\text{and}\quad  \frac1p\geq \sum_{j=1}^N \frac{1}{p_j}.
\]
\end{proposition}

\begin{proof}
Obviously $\ 0<\frac1u= \sum_{j=1}^N \frac{1}{u_j}\leq \sum_{j=1}^N \frac{1}{p_j}\leq \frac1p$ such that $\ms$ is well-defined for $0<p\leq u<\infty$. The rest is iterated application of H\"older's inequality.
\end{proof}

\begin{corollary}
Let $0< p\leq u<\infty$, $\lambda\in\ms$, and $0<r<\infty$. Then $|\lambda|^r=\{|\lambda_k|^r\}_{k\in\Zn} \in m_{u/r, p/r}$ with
\[
\left\| \, |\lambda|^r\ | m_{\frac{u}{r}, \frac{p}{r}}\right\| = \| \lambda| \ms\|^r.
\]
\end{corollary}

\begin{proof}
This follows by the definition immediately.
\end{proof}

\begin{proposition}
Let $0<p<u\leq\infty$. Then $\ms$ is non-separable.
\end{proposition}

\begin{proof}
If $u=\infty$, then $m_{\infty,p}=\ell_\infty$ in view of Remark~\ref{rem-u=infty} and the result is well-known. So assume $0<p<u<\infty$ now. We adapt the arguments of \cite[Prop. 3.7]{RoTri2} appropriately. We consider a sequence of disjoint dyadic cubes $\{Q_{-j_\ell, m_\ell}\}_{\ell\in\nat}$, where $m_\ell\in\Zn$ and $j_\ell\in\no$ with $j_1<j_2< \cdots$. Let the sequence $\lambda=\{\lambda_k\}_{k\in\Zn}$ be given by
\begin{equation}\label{ms-13}
\lambda_k = \begin{cases} \pm 2^{-j_\ell \frac{d}{u}} & \text{if}\quad Q_{0,k}\subset Q_{-j_\ell,m_\ell}, \\ 0 & \text{otherwise}.
\end{cases}
\end{equation}
Then $\lambda\in\ms$, as can be seen as follows. Let $Q_{-j,m}$ be a dyadic cube.  If $Q_{-j,m}$ has empty intersection with  $\bigcup_\ell Q_{-j_\ell, m_\ell}$, then it is nothing to estimate. If the intersection is not empty, then either $Q_{-j,m}$ is contained  in exactly one cube $Q_{-j_\ell, m_\ell}$ or it contains  finitely many cubes  $Q_{-j_\ell, m_\ell}$. 

In the first case, if  $Q_{-j,m}\subset Q_{-j_\ell, m_\ell}$, then  $j\le j_\ell$.       
Therefore
\begin{align*}
\sup_{m\in\Zn} |Q_{-j,m}|^{\frac1u-\frac1p} \Big(& \sum_{k:Q_{0,k}\subset Q_{-j,m}}  |\lambda_k|^p\big)^{\frac1p} \\
&\leq \sup_{m\in\Zn} |Q_{-j,m}|^{\frac1u-\frac1p}\ 2^{-j_\ell \frac{d}{u}} \ |Q_{-j,m}|^{\frac1p} = 2^{-\frac{d}{u}(j_\ell-j)}\le 1.
\end{align*}
In the second case for appropriate $\ell_1<\ldots <\ell_\nu$ and $m_{\ell_i}\in\Zn$, $i=1,\ldots ,\nu$, there is a finite family $Q_{-j_{\ell_i}, m_{\ell_i}}$ such that $\bigcup_{i=1}^\nu Q_{-j_{\ell_i}, m_{\ell_i}}\subset Q_{-j,m}$. Now $j_{\ell_1}<\ldots <j_{\ell_\nu}\le j$. So
\begin{align*}
  \sup_{m\in\Zn}  |Q_{-j,m}|^{\frac1u-\frac1p} & \Big(\sum_{k:Q_{0,k}\subset Q_{-j,m}} \!\!\!|\lambda_k|^p\Big)^{\frac1p} \\
       & \le \sup_{m\in\Zn}  |Q_{-j,m}|^{\frac1u-\frac1p} \Big(\sum_{i=1}^\nu \sum_{k:Q_{0,k}\subset Q_{-j_{\ell_i}, m_{\ell_i}}}\!\!\!|\lambda_k|^p\Big)^{\frac1p} \\ &
     \le \sup_{m\in\Zn} |Q_{-j,m}|^{\frac1u-\frac1p} \Big(\sum_{i=1}^\nu 2^{j_{\ell_i}d(1-\frac{p}{u})}\Big)^{\frac1p} \\
  &\leq C \sup_{m\in\Zn} |Q_{-j,m}|^{\frac1u-\frac1p} 2^{-j_{\ell_\nu} d(\frac{1}{u}-\frac{1}{p})}  = C  2^{d(\frac{1}{u}-\frac1p)(j-j_{\ell_\nu})}\le C.
\end{align*}
Hence $\|\lambda|\ms\| \leq C $. 

Now consider two sequences $\lambda^1 =\{\lambda^1_k\}_{k\in\Zn}$ and 
$\lambda^2 =\{\lambda^2_k\}_{k\in\Zn}$ which are different. Therefore there exists at least one $k_0\in\Zn$ such that $\lambda^1_{k_0}\neq \lambda^2_{k_0}$. By construction this means $|\lambda^1_{k_0} -\lambda^2_{k_0}| = 2\ 2^{-j_{\ell_0} \frac{d}{u}}$ for some appropriate $\ell_0\in\no$ and $m_{\ell_0}\in\Zn$ with $Q_{0,k_0}\subset Q_{-j_{\ell_0}, m_{\ell_0}}$. Thus
\[
\|\lambda^1-\lambda^2| \ms\| \geq  2\ 2^{-j_{\ell_0} \frac{d}{u}} = 2^{1-j_{\ell_0}\frac{d}{u}}.
\]
Since the set of all admitted sequences $\lambda$ in \eqref{ms-13} is non-countable, having the cardinality of $\R$, $\ms$ is not separable. 
\end{proof}

\begin{ownremark}
Let us mention briefly that in \cite{GuKiSch} further (one-dimensional) approaches to weak and generalised Morrey sequence spaces were considered.
\end{ownremark}

\section{Embeddings}\label{sec-2}
We prove our main result about embeddings of different Morrey sequence spaces. Here we also use and adapt some ideas of our paper \cite{HaSk-bm1}.

\begin{theorem}  \label{cont}
Let    $0<p_1\leq u_1<\infty$ and  $0<p_2\leq u_2<\infty$.  Then the embedding 
\begin{equation} \label{ms1cont}
	 \mse\hookrightarrow \msz
\end{equation}
is continuous  if, and only if, 
the following conditions hold
\begin{equation}\label{ms3acont}
u_1 \le u_2\qquad \text{and}\qquad \frac{p_2}{u_2} \le\frac{p_1}{u_1}.
\end{equation}
The embedding \eqref{ms1cont} is never compact.
\end{theorem}


\begin{proof}
\noindent{\em Step 1}.~ First we prove the sufficiency of the conditions. 
If $u_1=u_2$, then  $p_2\le p_1$, and by the H\"older inequality we get  for any $j\in\no$ and $m\in\Zn$,
\begin{equation}
|Q_{-j,m}|^{-\frac{1}{ p_2}}\Big(\sum_{k:Q_{0,k}\subset Q_{-j,m}}\!\!\!|\lambda_{k}|^{p_2}\Big)^{\frac{ 1}{ p_2}} \le
|Q_{-j,m}|^{-\frac{1}{ p_1}}\Big(\sum_{k:Q_{0,k}\subset Q_{-j,m}}\!\!\!|\lambda_{k}|^{p_1}\Big)^{\frac{ 1}{ p_1}} .	
\label{ms-1}
\end{equation}
So \eqref{ms1cont} holds for  $u_1 = u_2$.  
 
Let now $u_1 < u_2$.  It should be clear that it is sufficient to consider sequences $\big\{\lambda_{m}\big\}_{m}$ satisfying the assumption
$$
\|\lambda|\mse\| = \sup_{j\in\no; m\in \Zn}\! |Q_{-j,m}|^{\frac{ 1}{ u_1} - \frac{ 1}{ p_1} }\Big(\sum_{k:Q_{0,k}\subset Q_{-j,m}}\!\!\!|\lambda_{k}|^{p_1}\Big)^{\frac{ 1}{ p_1}} = 1  . 
$$   
In that case  $|\lambda_{k}|\le 1$ for any $k\in\Zn$.  Let  $\frac{p_2}{u_2} = \frac{p_1}{u_1}$, i.e., $p_1<p_2$.  Then this leads for any $j\in\no$ and $m\in\Zn$ to

\begin{eqnarray}
	\sum_{k:Q_{0,k}\subset Q_{-j,m}}\!\!\!|\lambda_{k}|^{p_2} \le  	\sum_{k:Q_{0,k}\subset Q_{-j,m}}\!\!\!|\lambda_{k}|^{p_1} \le
	|Q_{-j,m}|^{1- \frac{ p_1}{ u_1}} = |Q_{-j,m}|^{1- \frac{ p_2}{ u_2}}. \nonumber 	   
\end{eqnarray}
The last inequality implies that
\begin{align}
  \|\lambda|\msz\| &=  \sup_{j\in\no; m\in \Zn}\! |Q_{-j,m}|^{\frac{ 1}{ u_2} - \frac{ 1}{ p_2} }\Big(\sum_{k:Q_{0,k}\subset Q_{-j,m}}\!\!\!|\lambda_{k}|^{p_2}\Big)^{\frac{ 1}{ p_2}} \nonumber\\
  & \le\  1=\|\lambda|\mse\|. 
\label{ms-2}
\end{align}

Combining the case $\frac{p_2}{u_2} = \frac{p_1}{u_1}$ with  the case $u_1 = u_2$ we can easily prove  \eqref{ms1cont} for general values of $p_2$ and $p_1$. Indeed, if  $p=\frac{p_1u_2}{u_1}$, then by \eqref{ms-2} applied to $\mse\hookrightarrow m_{u_2,p}$ we obtain
\begin{align}
  \sup_{j\in\no; m\in \Zn}\! |Q_{-j,m}|^{\frac{ 1}{ u_2} - \frac{ 1}{ p}}& \Big(\sum_{k:Q_{0,k}\subset Q_{-j,m}}\!\!\!|\lambda_{k}|^{p}\Big)^{\frac{ 1}{ p}}  \nonumber\\
  \le &
\sup_{j\in\no; m\in \Zn}\! |Q_{-j,m}|^{\frac{ 1}{ u_1} - \frac{ 1}{ p_1}}\Big(\sum_{k:Q_{0,k}\subset Q_{-j,m}}\!\!\!|\lambda_{k}|^{p_1}\Big)^{\frac{ 1}{ p_1}} . \label{ms-3}
\end{align}  
On the other hand, using \eqref{ms-1} applied to $p_2\leq p$, we arrive at 
\begin{align}
  \sup_{j\in\no; m\in \Zn}\!|Q_{-j,m}|^{\frac{ 1}{ u_2} - \frac{ 1}{ p_2} }& \Big(\sum_{k:Q_{0,k}\subset Q_{-j,m}}\!\!\!|\lambda_{k}|^{p}\Big)^{\frac{ 1}{ p_2}} \nonumber\\
  \le &
\sup_{j\in\no; m\in \Zn}\! |Q_{-j,m}|^{\frac{ 1}{ u_2} - \frac{ 1}{ p} }\Big(\sum_{k:Q_{0,k}\subset Q_{-j,m}}\!\!\!|\lambda_{k}|^{p}\Big)^{\frac{ 1}{ p}}.	
\label{ms-4}
\end{align}
Thus \eqref{ms-4} and \eqref{ms-3} imply   \eqref{ms1cont}. \\

\noindent{\em Step 2}.~   We consider the  necessity of the conditions.

{\em Substep 2.1.}~ First we assume that  $u_2 < u_1$. For any $j\in \N$ we put $m_j=(2^{2j}, 0,\ldots , 0)\in \Z^d$. We put 
\[
\lambda_{k} =\,\begin{cases}
|Q_{-j,m_j}|^{-\frac{1}{u_1}}& \text{if}\quad Q_{0,k}\subset Q_{-j,m_j}\\
0 & \text{otherwise}.
\end{cases}
\]
Then straightforward calculation shows that  $\|\lambda|\mse\| = 1$. On the other hand, 
\[ \sup_{j\in\no} |Q_{-j,m_j}|^{\frac{1}{u_2}-\frac{1}{p_2}} \Big(\sum_{k:Q_{0,k}\subset Q_{-j,m_j}}\!\!\! |\lambda_k|^{p_2}\Big)^{1/p_2} = \sup_{j\in\no}\, |Q_{-j,m_j}|^{\frac{1}{u_2}-\frac{1}{u_1}}= \infty .\]
So $\{\lambda_k\}_k$ does not belong to $\msz$.

{\em Substep 2.2.}~ Now we assume that  $u_1\leq u_2$ and  $\frac{p_1}{u_1} < \frac{p_2}{u_2}$ , in particular, $\frac{p_1}{u_1} < 1$. For any $j\in \N$ we put 
$$k_j = \whole{2^{dj(1 - \frac{p_1}{u_1})}},$$
recall $\whole{x}=\max\{l\in\Z: l \leq x\}$. Then $1\le k_\nu < 2^{d\nu}$ and 
\begin{equation}\label{30.12-1}
	k_j \le \ c_{p_1 q_1}\ 2^{d(j - \nu)} k_\nu, \qquad \text{if}\qquad 1\le \nu < j \, .
\end{equation}
For convenience let us assume that $c_{p_1 q_1}= 1$ (otherwise the argument below has to be modified in an obvious way). 
For any  $j\in \N$ we define a sequence $\lambda^{(j)} = \big\{\lambda^{(j)}_{k}\big\}_{k}$ in the following way. We assume that $k_\nu$ elements of the sequence  equal  $1$ and the rest is equal to  $0$.  If $Q_{0,k}\nsubseteq Q_{-j,0}$, then we put $\lambda^{(j)}_{k}= 0$. Moreover,  because of the inequality \eqref{30.12-1}, we can choose the elements that equal $1$ is such a way that the following property holds
\begin{align*}  
&\text{if}\quad Q_{-\nu,\ell}\subseteq Q_{-j,0} \quad \text{and} \quad Q_{-\nu,\ell} = \bigcup_{i=1}^{2^{d\nu}} Q_{0,m_i}, \\ 
&\text{then at most}\; k_\nu\; \text{elements}\; \lambda^{(j)}_{0,k_i}\; \text{equal}\; 1.  
\end{align*}
By construction, if $Q_{-\nu,\ell}\subseteq Q_{-j,0}$, then  
\begin{align}	
	\sum_{k:Q_{0,k}\subset
          Q_{-\nu,\ell}}\!\!\!|\lambda^{(j)}_{0,k}|^{q_1} \le k_\nu \le 2^{d\nu(1 - \frac{p_1}{u_1})}
\label{30.12-2}
\end{align}
and  the last sum is equal to $k_\nu$ if $\nu=j$. Thus 
\begin{equation}\label{30.12-3}
	\|\lambda^{(j)}|\mse\|\,\le \, 1 .
\end{equation}
Furthermore, the assumption $0<\frac{p_1}{u_1} < \frac{p_2}{u_2}\le 1$ implies that 
$$ 
	\frac{\whole{2^{dj(1 - \frac{p_1}{u_1})}}}{2^{dj(1 - \frac{p_2}{u_2})}}\, \longrightarrow \infty \qquad \text{if}\qquad j \rightarrow \infty\, .
$$ 
So for any $N\in \N$ there exists a number $j_N\in\nat$ such that 
$$
  N  2^{dj_N(1 - \frac{p_2}{u_2})} \le \whole{2^{dj_N(1 - \frac{p_1}{q_1})}} = k_{j_N} =  	\sum_{k:Q_{0,k}\subset
          Q_{-j_N,0}}\!\!\!|\lambda^{(j_N)}_{0,k}|^{p_2} . 
$$
But this immediately implies that  
\begin{equation}\label{30.12-4}
	N^{1/p_2}\, \le \, \|\lambda^{(j_N)}|\msz\|\, .
\end{equation}

However, since we assume that the embedding \eqref{ms1cont} holds, there is a positive constant $c>0$ such that    
$$
\|\lambda^{(j)}|\msz\| \,\le \, c\, 	\|\lambda^{(j)}|\mse\|\,\le \, c \qquad \text{for any}\qquad \lambda^{(j)}\in \mse\, .
$$
In view of the last inequalities  \eqref{30.12-3} and \eqref{30.12-4} we get 
$$
	N^{1/p_2}\, \le \, \|\lambda^{(j_N)}|\msz\|\,   \le C  \|\lambda^{(j_N)}|\mse\|\, \leq \, C \, , 
$$
and this leads to a contradiction for large  $N$.\\

\noindent{\em Step 3}.~ The non-compactness of \eqref{ms1cont} immediately follows from Proposition~\ref{Prop-1-ms},
\[
\ell_{u_1} \hookrightarrow \mse \hookrightarrow \msz\hookrightarrow \ell_\infty
\]
and the non-compactness of $\ell_{u_1}\hookrightarrow \ell_\infty$.
\end{proof}

\begin{corollary}\label{Coro-ms-1}
Let    $0<p_1\leq u_1<\infty$ and  $0<p_2\leq u_2<\infty$. Then 
$\mse= \msz$ (in the sense of equivalent norms) if, and only if,  $u_1=u_2$ and $p_1=p_2$.
\end{corollary}

\begin{proof}
This follows immediately from Theorem~\ref{cont}. 
\end{proof}

\begin{corollary}\label{Coro-ms-2}
Let    $0<p_1\leq u_1<\infty$ and  $0<p_2\leq u_2<\infty$.
\bli
\item[{\bfseries\upshape (i)}]
Then $\mse \hookrightarrow \ell_{u_2}$ if, and only if, $u_1\leq u_2$ and $p_1=u_1$, that is, if, and only if, $\mse=\ell_{u_1}$ and $\ell_{u_1}\hookrightarrow \ell_{u_2}$.
\item[{\bfseries\upshape (ii)}]
Then $\ell_{u_1} \hookrightarrow \msz$ if, and only if, $u_1\leq u_2$, that is, if, and only if, $\ell_{u_1}\hookrightarrow \ell_{u_2}$.
\eli
\end{corollary}

\begin{proof}
This follows immediately from Theorem~\ref{cont}. 
\end{proof}

\begin{ownremark}
Let us mention the following essential feature: if $0<p<u<\infty$, that is, we are in the proper Morrey situation, then there is never an embedding into any space $\ell_r$ whenever $0<r<\infty$, but we always have $\ms \hookrightarrow \ell_\infty$, in view of Proposition~\ref{Prop-1-ms}(iii) and Corollary~\ref{Coro-ms-2}(i).
\end{ownremark}

\begin{ownremark}
We briefly want to compare Theorem~\ref{cont} with forerunners in 
\cite{GuKiSch} and in the parallel setting of Morrey function spaces.

In \cite[Prop.~2.4]{GuKiSch} the one-dimensional counterpart of Theorem~\ref{cont} can be found in the case when (in our notation) $1\leq p_2\leq p_1\leq u_1=u_2<\infty$, with some discussion about the sharpness of that result. Obviously condition \eqref{ms3acont} is automatically satisfied in this case. The method of their proof in \cite{GuKiSch} is different from ours. 

We turn to function spaces and first consider spaces $\M(Q)$ defined on a cube $Q$, where \eqref{morrey-func} has to be adapted appropriately. Then by a result of Piccinini in \cite{Pic69}, see also \cite{PicPhD}, for $0<p_i\leq u_i<\infty$, $i=1,2$, 
\[
\Me(Q)  \hookrightarrow \Mz(Q)\qquad\text{if, and only if,}\qquad
p_2\leq p_1 \quad\text{and}\quad u_2\leq u_1.
\]
This result was extended to $\Rn$ by Rosenthal in \cite{rosenthal}, reading as 
\[
\Me(\Rn) \hookrightarrow \Mz(\Rn) \qquad\text{if, and only if,}\qquad p_2\leq p_1\leq u_1=u_2.
\]
So a similar diversity as in the classical $L_p$-setting (spaces on bounded domains versus $\Rn$ versus sequence spaces $\ell_p$) is obvious.

What is, however, more surprising is the similarity with our result  \cite[Thm.~3.2]{HaSk-bm1}  in the context of sequence spaces $\mb$ appropriate for smoothness Morrey spaces. In the limiting situation $s_1-\frac{d}{u_1}=s_2-\frac{d}{u_2}$ (and in adapted notation) we have shown that
\[
\mbe \hookrightarrow \mbz
\]
if, and only if, \eqref{ms3acont} and $q_1\leq q_2$. 
\end{ownremark}

\section{The pre-dual of $\ms$}\label{sec-3}

Results concerning (pre-)dual spaces in the setting of Morrey function spaces have some history, we refer to \cite{alvarez,kalita,zorko}, and, more recently, to \cite{AX04} in this respect. We rely on the paper \cite[Sect. 4]{RoTri2} where also further discussion can be found.

\begin{definition}\label{level-seq-X}
Let $1\le p< u <\infty$ and $\frac{1}{p}+\frac{1}{p'}=1$, as usual.  For any $j\in\no$ we define ${\mathcal X}^{(j)}_{u,p}={\mathcal X}^{(j)}_{u,p}(\Zn) $ by
\begin{align}
  {\mathcal X}^{(j)}_{u,p}(\Zn) = \Big\{ \lambda= &\{\lambda_k\}_{k\in \Zn}\subset\C :\nonumber\\
&   \|\lambda\|^{(j)}_{u,p}  = 2^{jd(\frac{1}{p}-\frac{1}{u})}
\sum_{m\in \Zn}\Big(
\sum_{k:Q_{0,k}\subset Q_{-j,m}}
 \!\!\!|\lambda_k|^{p'} \Big)^{\frac{1}{p'}} < \infty 
\Big\}. \label{x-pi}
\end{align}
\end{definition}

\begin{lemma}\label{predual-l1}
Let $1\le p< u <\infty$ and $j\in\no$. Then
\begin{equation}\label{ms-5}
{\mathcal X}^{(j)}_{u,p}(\Zn)  =  \ell_{1}(\Zn)
\end{equation}
(in the sense of equivalent norms). Moreover, the embedding 
\[ \id_j : {\mathcal X}^{(j)}_{u,p}(\Zn)  \hookrightarrow  \ell_{u'}(\Zn) 
\]
satisfies
\begin{equation}\label{ms-6}
\left\|\id_j : {\mathcal X}^{(j)}_{u,p}(\Zn)  \hookrightarrow  \ell_{u'}(\Zn)\right\|=1.
\end{equation}
\end{lemma}

\begin{proof}
We begin with \eqref{ms-5}. Note that 
\[
\Big(\sum_{k:Q_{0,k}\subset Q_{-j,m}}
 \!\!\!|\lambda_k|^{p'} \Big)^{\frac{1}{p'}} \leq \sum_{k:Q_{0,k}\subset Q_{-j,m}}
 \!\!\!|\lambda_k| \ \leq \ 2^{j\frac{d}{p}}\ \Big(\sum_{k:Q_{0,k}\subset Q_{-j,m}}
 \!\!\!|\lambda_k|^{p'} \Big)^{\frac{1}{p'}}. 
\]
Hence, by definition,
\begin{align}
\|\lambda\|^{(j)}_{u,p} & = 2^{jd(\frac{1}{p}-\frac{1}{u})}
\sum_{m\in \Zn}\Big(\sum_{k:Q_{0,k}\subset Q_{-j,m}}
\!\!\!|\lambda_k|^{p'} \Big)^{\frac{1}{p'}}\nonumber\\
&\geq 
2^{jd(\frac{1}{p}-\frac{1}{u})} 
\sum_{m\in \Zn}\Big(\sum_{k:Q_{0,k}\subset Q_{-j,m}}
 \!\!\!|\lambda_k| \Big) 2^{-j\frac{d}{p}} \nonumber \\
&= 2^{-j\frac{d}{u}} \sum_{k\in\Zn} |\lambda_k| = 2^{-jd\frac1u} \|\lambda|\ell_1\|,\label{ms-7}
\end{align}
hence ${\mathcal X}^{(j)}_{u,p}  \hookrightarrow  \ell_1$. Conversely,
\begin{align}
\|\lambda\|^{(j)}_{u,p} & = 2^{jd(\frac{1}{p}-\frac{1}{u})}
\sum_{m\in \Zn}\Big(\sum_{k:Q_{0,k}\subset Q_{-j,m}}
\!\!\!|\lambda_k|^{p'} \Big)^{\frac{1}{p'}} \nonumber\\
& \leq 
2^{jd(\frac{1}{p}-\frac{1}{u})} 
\sum_{k\in\Zn} |\lambda_k| = 2^{jd(\frac{1}{p}-\frac{1}{u})} \|\lambda|\ell_1\|,\label{ms-8}
\end{align}
which results in $\ell_1\hookrightarrow {\mathcal X}^{(j)}_{u,p}$ and thus finishes the proof of \eqref{ms-5}.\\
Similar to \eqref{ms-7} we obtain for the embedding $\id_j$,
\begin{align}
\|\lambda\|^{(j)}_{u,p} & = 2^{jd(\frac{1}{p}-\frac{1}{u})}
\sum_{m\in \Zn}\Big(\sum_{k:Q_{0,k}\subset Q_{-j,m}}
\!\!\!|\lambda_k|^{p'} \Big)^{\frac{1}{p'}}\nonumber\\
&\geq 
2^{jd(\frac{1}{p}-\frac{1}{u})} 
\sum_{m\in \Zn}\Big(\sum_{k:Q_{0,k}\subset Q_{-j,m}}
\!\!\! |\lambda_k|^{u'} \Big)^{\frac{1}{u'}} 2^{-jd(\frac{1}{u'}-\frac{1}{p'})} \nonumber \\
&\geq\ 2^{jd(\frac1p-\frac1u-\frac1p+\frac1u)} \Big(\sum_{k\in\Zn} |\lambda_k|^{u'}\big)^{\frac{1}{u'}} = \|\lambda|\ell_{u'}\|,\label{ms-9}
\end{align}
since $p'>u'>1$. Thus $\|\id_j\|\leq 1$. Now let $m_0\in\Zn$ be fixed and consider $\lambda^0=\{\lambda_k^0\}_{k\in\Zn}$ given by
\[
\lambda_k^0 = \begin{cases} 2^{-j\frac{d}{u'}}, & \text{if}\ Q_{0,k}\subset Q_{-j,m_0}, \\ 0, & \text{otherwise}.
\end{cases}
\]
Thus $\|\lambda^0 | \ell_{u'}\| = 2^{-j\frac{d}{u'}} |Q_{-j,m_0}|^{\frac{1}{u'}}=1$, and
\begin{align*}
\|\lambda^0\|^{(j)}_{u,p}  & = 2^{jd(\frac{1}{p}-\frac{1}{u})}
\sum_{m\in \Zn}\Big(\sum_{k:Q_{0,k}\subset Q_{-j,m}}
\!\!\! |\lambda_k|^{p'} \Big)^{\frac{1}{p'}} = 
2^{jd(\frac{1}{p}-\frac{1}{u})} 2^{-j\frac{d}{u'}} |Q_{-j,m_0}|^{\frac{1}{p'}}\\
&= 2^{jd(\frac1p-\frac1u-\frac{1}{u'}+\frac{1}{p'})} = 1,
\end{align*}
such that finally $\|\id_j\| \geq 1$. This completes the proof of \eqref{ms-6}.
\end{proof}

Now we combine the above sequence spaces ${\mathcal X}^{(j)}_{u,p}(\Zn)$ at level $j\in\no$ as follows.

\begin{definition}
Let $1\le p< u<\infty$. We define $\xup=\xup(\Zn)$ by
\begin{align}
{\mathcal X}_{u,p} (\Zn) = \bigg\{\lambda & =\{\lambda_k\}_{k\in\Zn}\subset\C: \text{for any $j\in\no$ there exists}\ \lambda^{(j)}\in\xjup(\Zn)\ \nonumber\\& \text{such that}\quad \lambda=\sum_{j=0}^\infty \lambda^{(j)}, \quad\text{and}\quad \sum_{j=0}^\infty 
\|\lambda^{(j)}\|^{(j)}_{u,p} < \infty 
\bigg\},\label{x_up}
\end{align}
equipped with the norm 
\begin{equation}\label{x_up_norm}
\|\lambda| \xup \| = \inf \sum_{j=0}^\infty \|\lambda^{(j)}\|^{(j)}_{u,p} \ , 
\end{equation}   
where the infimum is taken over all admitted decompositions of $\lambda$ according to \eqref{x_up}. 
\end{definition}
 
\begin{proposition}\label{predual-l2-2a}
Let $1\leq p<u<\infty$. 
\bli
\item[{\bfseries\upshape (i)}]
Then
\begin{equation}\label{ms-10}
\ell_{1}(\Zn) \hookrightarrow \xup(\Zn) \hookrightarrow \ell_{u'}(\Zn), 
\end{equation}
and 
\begin{equation}\label{ms-11}
\left\|\id : \ell_{1}(\Zn) \hookrightarrow \xup(\Zn)\right\|=\left\|\id: \xup(\Zn) \hookrightarrow \ell_{u'}(\Zn)\right\| = 1.
\end{equation}
\item[{\bfseries\upshape (ii)}] 
Let for $n\in \Zn$,  $e^{(n)} = \{e^{(n)}_k\}_{k\in\Zn}$ be given by
\[ 
e^{(n)}_k = \begin{cases} 1 & \text{if} \quad n=k , \\ 0 & \text{otherwise}. \end{cases}
\] 
The system $\{e^{(n)}\}_{n\in\Zn}$ forms a normalised unconditional basis in $\xup(\Zn)$.
\item[{\bfseries\upshape (iii)}] 
$\xup(\Zn)$ is a separable Banach space. 
\eli
\end{proposition}

\begin{proof}
\noindent{\em Step 1}.~ Let $\lambda\in\ell_1$. Then Lemma~\ref{predual-l1} applied with $j=0$, in particular \eqref{ms-5}, imply that we obtain an admitted representation of $\lambda$ in \eqref{x_up} choosing $\lambda^{(0)}=\lambda\in {\mathcal X}^{(0)}_{u,p} $ and $\lambda^{(j)}=0$, $j\in\nat$. This ensures $\lambda\in\xup$ and, in view of \eqref{x_up_norm} and \eqref{ms-8},
\[
\|\lambda|\xup\| \leq \|\lambda\|^{(0)}_{u,p}\leq \|\lambda|\ell_1\|.
\]
Thus $\|\id : \ell_{1}(\Zn) \hookrightarrow \xup(\Zn) \|\leq 1$. 
If $\lambda\in\xup$, then there exists a decomposition according to \eqref{x_up} and we can conclude
\[
\|\lambda|\ell_{u'}\| \leq \sum_{j=0}^\infty \left\|\lambda^{(j)}|\ell_{u'}\right\|\leq \sum_{j=0}^\infty \|\lambda^{(j)} \|^{(j)}_{u,p},
\]
where we applied \eqref{ms-6}. Talking the infimum over all possible representations according to \eqref{x_up} we get by \eqref{x_up_norm} that
\[
\|\lambda|\ell_{u'}\| \leq \|\lambda|\xup\|, \quad\text{and thus}\quad 
\left\|\id: \xup(\Zn) \hookrightarrow \ell_{u'}(\Zn)\right\|\leq 1.
\]
To complete the proof of (i), we have to show the converse inequalities in \eqref{ms-11}. However, $\xjup\hookrightarrow \xup$ for any $j\in\no$ with $\|\id:\xjup\hookrightarrow \xup\|\leq 1$, such that \eqref{ms-6} yields
\[
1=\| \id: \xjup\hookrightarrow \ell_{u'}\|\leq \|\id:\xup\hookrightarrow \ell_{u'}\|,
\]
confirming the latter equality in \eqref{ms-11}. Now we are done, since
\[
1=\|\id:\ell_1\hookrightarrow \ell_{u'}\|\leq \|\id:\ell_1\hookrightarrow \xup\| \ \|\id:\xup\hookrightarrow \ell_{u'}\| = \|\id:\ell_1\hookrightarrow \xup\|.
\]

\noindent{\em Step 2}.~
Concerning (ii), one can easily calculate that $\|e^{(n)}|\xup\|=1$ and that the system is complete  in $\xup$. So the statement follows from the trivial inequality 
\[ \Big\|\sum_{|n|\leq \ell} \varepsilon_{n}\lambda_n e^{(n)}|\xup\Big\| \le \Big\|\sum_{|n|\leq \ell} \lambda_n e^{(n)}|\xup\Big\|,\]    
$\varepsilon_{n} = \pm 1$,  $\lambda_n\in \C$,  cf. eg. \cite[Theorem 6.7]{Heil}. \\

\noindent{\em Step 3}.~ The proof of (iii) is standard.
\end{proof}

\begin{proposition}\label{xup-predual}
Let $1\le p < u <\infty$. Then ${\mathcal X}_{u,p}(\Z^d)$ is a pre-dual space of $m_{u,p}(\Z^d)$.  
\end{proposition}

\begin{proof}
Let $\mu\in m_{u,p}$,  $\lambda\in {\mathcal X}_{u,p}(\Z^d)$ and let $\lambda = \sum_{j=0}^{\infty} \lambda^{(j)}$, $\lambda^{(j)}\in {\mathcal X}^{(j)}_{u,p}$. Then
\begin{align*}
\Big|\sum_{k\in\Zn} \lambda_k\mu_k\Big|  & \le \sum_{j=0}^\infty \sum_{k\in\Zn} |\lambda_k^{(j)} \mu_k| \\
&= \sum_{j=0}^\infty \sum_{m\in\Zn} \sum_{k: Q_{0,k}\subset Q_{-j,m}}\!\!\! |\lambda_k^{(j)} \mu_k| \\
&\leq 
\sum_{j=0}^\infty \sum_{m\in \Z^d} \Big(\sum_{k:Q_{0,k}\subset Q_{-j,m}} \!\!\!|\mu_{k}|^p\Big)^{\frac{1}{p}} \Big(\sum_{k:Q_{0,k}\subset Q_{-j,m}} \!\!\!|\lambda^{(j)}_{k}|^{p'}\Big)^{\frac{1}{p'}} \\
& \le  \|\mu| \ms\| \sum_{j=0}^\infty 2^{dj(\frac{1}{p}-\frac{1}{u})} \sum_{m\in \Z^d} \Big(\sum_{k:Q_{0,k}\subset Q_{-j,m}}\!\!\! |\lambda^{(j)}_{k}|^{p'}\Big)^{\frac{1}{p'}}\\ &\le  \|\mu| \ms\| \sum_{j=0}^\infty  \|\lambda^{(j)}\|^{(j)}_{u,p}. 
\end{align*} 
Taking the infimum over all representations of $\lambda$ we get 
\begin{equation}\label{dual1}
\Big|\sum_{k\in\Zn} \lambda_k\mu_k\Big| \le \|\mu| \ms\| \|\lambda|\xup\| . 
\end{equation}

On the other hand, if   $f\in (\xup(\Zn))'$, then
\begin{equation*}
|f(\lambda)| \le \|f\|\,  \|\lambda|\xup\|,
\end{equation*} 
where $\|f\|=\sup_{\|\lambda|\xup\|=1} |f(\lambda)|$, as usual. 
For any dyadic cube $Q_{-\nu,m}$, $\nu\in\no$, we take 
\[\lambda^{(\nu,m)} = \sum_{k:Q_{0,k}\subset Q_{-\nu,m}} \!\!\! 2^{\nu d(\frac{1}{p}-\frac{1}{u})}\lambda_k e^{(k)}.\] 
Then 
$\lambda^{(\nu,m)} \in \xup$ if, and only if,  
$\{\lambda_k\}_{k\in\Zn}\in \ell^{2^{\nu d}}_{p'}(Q_{-\nu,m})$ and  $\|\lambda^{(\nu,m)}|\xup\| \le  \|  \{\lambda_k\}_k|\ell^{2^{\nu d}}_{p'}(Q_{-\nu,m})\|$.  
Moreover, 
\begin{align*}
\Big| \sum_{k:Q_{0,k}\subset Q_{-\nu,m}} \!\!\!\lambda_k 2^{\nu d(\frac{1}{p}-\frac{1}{u})} f(e^{(k)})\Big| & =  \ 
\left| f(\lambda^{(\nu,m)})\right|\,  \le \, 
\|f\| \, \|\lambda^{(\nu,m)} | \xup\| \\
& \le  \|f\| \,  \| \{\lambda_k\}_k|\ell^{2^{d\nu}}_{p'}(Q_{-\nu,m})\|.
\end{align*} 
By duality for the $\ell_p$ spaces we get  $ \{2^{\nu d(\frac{1}{p}-\frac{1}{u})} f(e^{(k)})\}_k\in \ell^{2^{d\nu}}_p(Q_{-\nu,m})$ and 
\begin{align*}
  |Q_{-\nu,m}|^{\frac{1}{u}-\frac{1}{p}}\left(\sum_{k:Q_{0,k}\subset Q_{-\nu,m}}\!\!\!|f(e^{(k)})|^p\right)^{\frac{1}{p}}  & = \| \{2^{\nu d(\frac{1}{p}-\frac{1}{u})} f(e^{(k)})\}_k|\ell^{2^{d\nu}}_p(Q_{-\nu,m})\|\\
  &\le \|f\|\, .
\end{align*}
So $\{f(e^{(k)})\}_k \in \ms(\Zn)$  and $\|\{f(e^{(k)})\}_k|\ms\|\le \|f\|$. 
\end{proof}

Next we define a closed proper subspace of $\ms$ as follows. Let $c_{00}$ denote the finite sequences in $\C$, that is, sequences which possess only finitely many non-vanishing elements. We define  $\ms^{00}=\ms^{00}(\Zn)$ to be the closure of $c_{00}$ in $\ms$,
\[
\ms^{00} = \overline{c_{00}}^{\|\cdot|\ms\|}.
\]
Obviously $\ms^{00}$ is separable. We shall prove below that $\xup$ is the dual space of $\ms^{00}$.  We begin with some general properties. For that reason, let us denote by $\ms^0= \ms^0(\Zn)$ the subspace of null sequences which belong to $\ms$,
\[
\ms^0= \ms  \cap c_0\ . 
\]
Then we have the following basic properties.

\begin{lemma}
Let $0<p<u<\infty$. Then $\ms^0$ and $\ms^{00}$ are proper closed subspaces of $\ms$, with 
\[ \ms^{00}\subsetneq\ms^0 \subsetneq \ms . \]
\end{lemma}

\begin{proof}
By definition, $\ms^{00}$ is a closed subspace of $\ms$. The fact that $\ms^0 \subsetneq \ms$ is a proper subspace of $\ms$ follows from Proposition~\ref{Prop-1-ms}(v). We show that $\ms^0$ is closed in $\ms$,
\[
\ms^0 = \overline{\ms^0}^{\|\cdot|\ms\|}.
\]
Clearly $\ms^0 \subseteq \overline{\ms^0}^{\|\cdot|\ms\|}$, so we have to verify the converse inclusion. Let $\lambda=\{\lambda_k\}_{k\in\Zn}\in\overline{\ms^0}^{\|\cdot|\ms\|}$ and $\ve>0$ be arbitrary. Then, by definition, there exists some $\mu=\{\mu_k\}_{k\in\Zn}\in\ms^0$ such that
\begin{equation}\label{ms-12}
\|\mu-\lambda|\ms\| = \sup_{j\in\no, m\in\Zn} |Q_{-j,m}|^{\frac1u-\frac1p} \Big(\sum_{k:Q_{0,k}\subset Q_{-j,m}} \!\!\!|\mu_k-\lambda_k|^p\Big)^{\frac1p} < \ve.
\end{equation}
Now $\mu\in\ms^0\subset \ms$ and $\ms$ is complete, thus $\lambda\in\ms$. Moreover, applying \eqref{ms-12} with $j=0$ implies that 
\[
\|\mu-\lambda|\ell_\infty\| = \sup_{k\in\Zn} |\mu_k-\lambda_k|<\ve.
\]
However, $\mu\in\ms^0 \subset c_0$ thus leads to $\lambda\in c_0$. So finally $\lambda\in \ms \cap c_0 = \ms^0$.\\
It remains to verify that $\ms^{00}\subsetneq\ms^0$. First note that, by definition, $\ms^{00} \subseteq \ms^0$. Now consider special lattice points $m_j=(2^{2j},0,\ldots , 0)\in\Zn$, $j\in\no$, and put
\[
\lambda_{k} = \begin{cases}
2^{-j\frac{d}{u}} & \text{if}\quad Q_{0,k}\subset Q_{-j,m_j} \\
0 & \text{otherwise}.
\end{cases}
\]
Then $\lambda\in c_0\cap \ms=\ms^0$, but obviously $\lambda\notin \ms^{00}$. 
\end{proof}

\begin{ownremark} 
The above result sheds some further light on the difference of the two norms $\|\cdot|\ell_\infty\| $ and $\|\cdot|\ms\|$, since in the classical setting 
\[ \overline{c_{00}}^{\|\cdot|\ell_\infty\|} = c_0
\]
is well-known, in contrast to $\ms^{00} \subsetneq \ms^0$.
\end{ownremark}

We need the following lemma.

\begin{lemma}\label{predual-l3}
Let $0<p<u<\infty$, and $\lambda\in \ms^{00}(\Zn)$. Then there exists a dyadic cube $Q(\lambda)$ such that
	\[\|\lambda|\ms\| = |Q(\lambda)|^{\frac{1}{u}-\frac{1}{p}}\left(\sum_{k:Q_{0,k}\subset Q(\lambda)} \!\!|\lambda_k|^p\right)^{\frac{1}{p}}.  \] 
\end{lemma} 

\begin{proof}
For any sequence $\lambda\in \ms$ and any dyadic cube $Q$, we shall denote by $\lambda|_Q$ the restriction of $\lambda$ to $Q$, i.e., $(\lambda|_Q)_k = \lambda_k$ if $Q_{0,k}\subset Q$ and $(\lambda|_Q)_k = 0$ otherwise. Without loss of generality we may assume $\lambda\not\equiv 0$. 
	
	We choose a positive number $\varepsilon < 1- 2^{d(\frac{1}{u}-\frac{1}{p})}$. If $\lambda\in \ms^{00}$, then there exists a dyadic cube $\widetilde{Q}$ such that 
\begin{equation*}
\|\lambda - \lambda|_{\widetilde{Q}}|\ms\| \le \varepsilon \|\lambda|\ms\|. 
\end{equation*}   
If $Q\cap \widetilde{Q}= \emptyset$, then 
\begin{align*}
|Q|^{\frac{1}{u}-\frac{1}{p}}\left(\sum_{k:Q_{0,k}\subset Q} \!\!|\lambda_k|^p\right)^{\frac{1}{p}} &\leq
|Q|^{\frac{1}{u}-\frac{1}{p}}\left(\sum_{k:Q_{0,k}\subset Q}\!\! |\lambda_k- (\lambda|_{\widetilde{Q}})_k|^p\right)^{\frac{1}{p}} \\
& \le \varepsilon \|\lambda|\ms\|<\|\lambda|\ms\|.
\end{align*}
If $\widetilde{Q}\varsubsetneq Q$,  then 	
  \begin{align*}
    |Q|^{\frac{1}{u}-\frac{1}{p}}\left(\sum_{k:Q_{0,k}\subset Q} \!\!|\lambda_k|^p\right)^{\frac{1}{p}} \le & \ 
    2^{d(\frac{1}{u}-\frac{1}{p})} |\widetilde{Q}|^{\frac{1}{u}-\frac{1}{p}}\left(\sum_{k:Q_{0,k}\subset \widetilde{Q}} \!\!|(\lambda|_{\widetilde{Q}})_k|^p\right)^{\frac{1}{p}}  \\   &+|Q|^{\frac{1}{u}-\frac{1}{p}}\left(\sum_{k:Q_{0,k}\subset Q} \!\!|\lambda_k- (\lambda|_{\widetilde{Q}})_k|^p\right)^{\frac{1}{p}} \\
    \le & \left(2^{d(\frac{1}{u}-\frac{1}{p})}+\varepsilon \right) \|\lambda|\ms\|<\|\lambda|\ms\|
   \end{align*}
by the choice of $\ve$. Therefore 
    \begin{align*}
    \|\lambda|\ms\| = \max_{Q\subset \widetilde{Q}} |Q|^{\frac{1}{u}-\frac{1}{p}}\left(\sum_{k:Q_{0,k}\subset Q} \!\!|\lambda_k|^p\right)^{\frac{1}{p}}  
    \end{align*}       
and the lemma is proved.
\end{proof}

\begin{proposition}\label{dual-p}
	Let $1\le p < u <\infty$. The dual space to  $\ms^{00}(\Zn)$ is isometrically isomorphic to   $\xup(\Zn)$. 
\end{proposition}

\begin{proof}
By Proposition~\ref{xup-predual}, the space $\xup$ is the pre-dual space of $\ms$. So it is sufficient to show that any functional on  $\ms^{00}$ can be represented by some element of $\xup$ with the equality of norms. 
	
	First we prove that the space $\ms^{00}$ can be  isometrically embedded into  a closed subspace of an appropriate vector valued $c_0$ space. Let for $j\in \no$ and $m\in \Zn$,  $A_{j,m} = \ell_{p}(Q_{-j,m}, w_{j,m})$ be a weighted finite-dimensional $\ell_{p}$ space, equipped  with the norm  
\[ \|\gamma| A_{j,m}\| = \|\gamma_\nu\|^{(j,m)}_p = \left( \sum_{\nu=1}^{2^{jd}}|\gamma_\nu|^p |Q_{-j,m}|^{\frac{p}{u}-1}\right)^{\frac{1}{p}} = 2^{jd(\frac1u-\frac1p)} \left( \sum_{\nu=1}^{2^{jd}}|\gamma_\nu|^p \right)^{\frac{1}{p}},\] 
where $\gamma=\{\gamma_{\nu}\}_{\nu=1}^{2^{jd}}$.

The space $c_0(A_{j,m})$ is the space of all sequences $a = \{a^{(j,m)}\}_{j\in\no, m\in\Zn}$ with $a^{(j,m)}\in A_{j,m}$, $j\in\no$, $m\in\Zn$, and such that $\|a^{(j,m)}|A_{j,m}\|\rightarrow 0$ if $j+ |m|\rightarrow \infty$. We equip   $c_0(A_{j,m})$ with the usual norm, i.e.,  
$$\|a|c_0(A_{j,m})\| = \sup_{j\in\no, m\in\Zn}  \|a^{(j,m)}|A_{j,m}\| .$$ 

It is well known that the dual space of $c_0(A_{j,m})$ is $\ell_1(A'_{j,m})$, where $b = \{b^{(j,m)}\}_{j\in\no, m\in\Zn}\in \ell_1(A'_{j,m})$ means
 \[
 \|b|\ell_1(A'_{j,m})\| = \sum_{j\in\no, m\in\Zn}  \|b^{(j,m)}|A'_{j,m}\| <\infty.
 \]
 Moreover, 
 \[
(a,b) = \sum_{j\in\no, m\in\Zn} (a_{j,m},b_{j,m}) , 
 \]
 cf., e.g., \cite[Lemma 1.11.1]{Tri-int}.
 
 Let $\lambda=\{\lambda_k\}_{k\in\Zn}\in \ms^{00}$. We define the mapping 
 \begin{equation}
 T: \ms^{00}\ni \lambda \mapsto \{\lambda^{(j,m)}\}_{j\in\no, m\in\Zn}\in c_0(A_{j,m}) 
 \end{equation}
 by putting  $\lambda^{(j,m)}_k = \lambda_k$ if $Q_{0,k}\subset Q_{-j,m}$. 
 
An argument similar to that one used in the proof of Lemma \ref{predual-l3} shows that $\|T(\lambda)|A_{j,m}\| \rightarrow 0$ if $j\rightarrow \infty$ or/and $|m|\rightarrow \infty$. Furthermore, by construction, 
 \[ \|\{\lambda^{(j,m)}\}_{j,m} | c_0(A_{j,m})\| = \|\lambda| \ms \|  .\] 
So we can identify $\ms^{00}$ with a closed subspace of $c_0(A_{j,m})$.  

Let $f\in (\ms)'$. The above identification and the Hahn-Banach theorem imply that $f$ can be extended to a continuous linear functional $\widetilde{f}$ on  $c_0(A_{j,m})$ and $\|\widetilde{f}\|= \|f\|$.  But $\widetilde{f}$ has a representation of the form 
\begin{equation}\label{dual2}
\widetilde{f}(\{\lambda^{(j,m)}\}_{j,m}) = \sum_{j\in\no, m\in\Zn}\ \sum_{k:Q_{0,k}\subset Q_{-j,m}} \!\!\!\lambda^{(j,m)}_{k}\mu^{(j,m)}_{k} ,  
\end{equation}
and 
\begin{equation}\label{dual3}
\|\widetilde{f}\| = \sum_{j\in\no, m\in\Zn} |Q_{-j,m}|^{\frac{1}{p}-\frac{1}{u}}\left(\sum_{k:Q_{0,k}\subset Q_{-j,m}} \!\!\!|\mu^{(j,m)}_{k}|^{p'}\right)^{\frac{1}{p'}} .  
\end{equation}

If $\{\lambda^{(j,m)}\}_{j,m}= T (\lambda)$, then the sum \eqref{dual2} can be rearranged in the following way, 
\begin{equation}\label{dual4}
f(\lambda)= \widetilde{f}(T(\lambda)) = \sum_{j=0}^\infty\ \sum_{k\in\Zn} \lambda_{k}\mu^{(j)}_{k}, 
\end{equation}
where the sequence $\mu^{(j)}=\{\mu^{(j)}_k\}_{k\in\Zn}$ is given by $\mu^{(j)}_{k}=\mu^{(j,m)}_{k}$ if $Q_{0,k}\subset Q_{-j,m}$. Recall that for any $k$ there is exactly one cube $Q_{-j,m}$ of size $2^{jd}$ such that    $Q_{0,k}\subset Q_{-j,m}$. Moreover, \eqref{dual3}
reads as
\begin{equation*}
\|f\| = \|\widetilde{f}\| = \sum_{j=0}^\infty \
2^{jd(\frac{1}{p}-\frac{1}{u})}
\sum_{m\in \Zn}\Big(
\sum_{k:Q_{0,k}\subset Q_{-j,m}} \!\!\!\!2^{dj(\frac{1}{p}-\frac{1}{u})}
 |\mu^{(j)}_k|^{p'} \Big)^{\frac{1}{p'}} = \sum_{j=0}^\infty \|\mu^{(j)}\|_{u,p}^{(j)}
%
\end{equation*}
by \eqref{x-pi}. Hence definition \eqref{x_up} yields that the sequence $\mu= \sum_{j=0}^\infty \mu^{(j)}\in \xup$ and 
\[ f(\lambda) = \sum_{k\in\Zn}^\infty \lambda_k \mu_k \quad \text{with} \quad \|f\|= \|\mu|\xup\| .\]
\end{proof}

\begin{ownremark}
Similar calculations for the Morrey function spaces can be found in \cite{RoTri2}. Moreover, arguments similar to those used in the proof of Proposition~\ref{dual-p} show that 
\[\ell_{u}(\Zn) \hookrightarrow \ms^{00}(\Zn) \hookrightarrow c_0(\Zn) .\]
\end{ownremark}

\section{Pitt's compactness theorem}
Now we prove Pitt's theorem for the Morrey sequence spaces. We follow the approach presented in \cite{D} and \cite{FZ}. The original result reads as follows.

\begin{theorem}[\cite{Pitt}]
Let $1\le q<p<\infty$. Every bounded linear operator  from $\ell_p$ into $\ell_q$ or from $c_0$ into $\ell_q$ is compact.
\end{theorem}

We start with the following lemma that shows the similarity of $\ms^{00}(\Z^d)$ to $c_0$ if $p<u$. 

\begin{lemma}\label{pitt-lemma}
Let $0<p<u<\infty$ and $w^{(n)}$ be a sequence in $\ms^{00}(\Zn)$, which is weakly convergent to zero, $w_n\rightharpoondown 0$. Then for any $\lambda\in \ms^{00}(\Zn)$, 
\[
\limsup_{n\to\infty} \|\lambda + w^{(n)}|\ms(\Zn)\| = \max\big\{\|\lambda|\ms(\Zn)\|, \limsup_{n\to\infty} \|w^{(n)}|\ms(\Zn)\| \big\}\, .
\]
\end{lemma}

\begin{proof}  
\emph{Step 1.}
First we assume that the sequence $\lambda\in \ms^{00}$ is finite. The sequences    $\lambda + w^{(n)}$ and $w^{(n)}$ belong to $\ms^{00}$, therefore, according to Lemma \ref{predual-l3}, there exist  dyadic cubes $Q_n$ and $\widetilde{Q}_n$ such that 
\begin{align}\label{pitt-l1}
 \|\lambda + w^{(n)}|\ms\| & =  |Q_n|^{\frac{1}{u}-\frac{1}{p}}\left(\sum_{k:Q_{0,k}\subset Q_n} \!\!|\lambda_k+ w_k^{(n)}|^p\right)^{\frac{1}{p}} ,\\
 \label{pitt-l1a}
 \| w^{(n)}|\ms\| & =  |\widetilde{Q_n}|^{\frac{1}{u}-\frac{1}{p}}\left(\sum_{k:Q_{0,k}\subset \widetilde{Q}_n} \!\!| w_k^{(n)}|^p\right)^{\frac{1}{p}} .
\end{align}
By the definition of $\limsup$ there is always a subsequence of cubes $\{Q_{n_i}\}_i$ such that   
\begin{equation}\label{pitt-l1b}
\limsup_{n\to\infty} \|\lambda + w^{(n)}|\ms\| = \lim_{i\rightarrow \infty} |Q_{n_i}|^{\frac{1}{u}-\frac{1}{p}}\left(\sum_{k:Q_{0,k}\subset Q_{n_i}} \!\!|\lambda_k+ w_k^{(n_i)}|^p \right)^{\frac{1}{p}}. 
\end{equation}
The cubes $Q_n$ are  dyadic cubes of size at least one, therefore we may assume that the subsequence satisfies one of the following alternative conditions: 
\begin{align}\label{pitt-l-alt1}
(1) &\quad  \text{there exists a dyadic cube} \; Q\; \text{such that}\; Q_{n_i}\subset Q \; \text{for any}\; i , \\
(2) & \quad \lim_{i\rightarrow \infty} |Q_{n_i}|= \infty ,\label{pitt-l-alt2}\\ 
(3) & \quad \sup_i |Q_{n_i}|<\infty \quad \text{and} \quad  Q_{n_i}\cap Q_{n_j}=\emptyset \quad \text{if} \quad i\not= j .\label{pitt-l-alt3}
\end{align}  
A similar statement holds for the cubes $\widetilde{Q}_n$. 

Please note that the weak convergence of the sequence  $w^{(n)}$  to zero implies the uniform convergence to zero of the coordinates of  $w^{(n)}$   on any dyadic cube  $Q$.   In the next steps we shall denote by $\widetilde{Q}(\lambda)$ the dyadic cube that contains the support of $\lambda$. 

\emph{Substep 1.1}  We prove that   
\begin{equation}\label{pitt-l2}
\max\big\{\|\lambda|\ms\|, \limsup_{n\to\infty} \|w^{(n)}|\ms\| \big\} \le \limsup_{n\to\infty} \|\lambda + w^{(n)}|\ms\|\, .
\end{equation}
We have that 
\begin{align} \nonumber
\|\lambda|\ms\| & = 
|Q(\lambda)|^{\frac{1}{u}-\frac{1}{p}}\left(\sum_{k:Q_{0,k}\subset Q(\lambda)} \!\!|\lambda_k|^p\right)^{\frac{1}{p}} \\
& = 
\lim_{i\rightarrow \infty} |Q(\lambda)|^{\frac{1}{u}-\frac{1}{p}}\left(\sum_{k:Q_{0,k}\subset Q(\lambda)} \!\!|\lambda_k+ w_k^{(n_i)}|^p\right)^{\frac{1}{p}} \nonumber\\
&\le \lim_{i\rightarrow \infty} |Q_{n_i}|^{\frac{1}{u}-\frac{1}{p}}\left(\sum_{k:Q_{0,k}\subset Q_{n_i}} \!\!|\lambda_k+ w_k^{(n_i)}|^p\right)^{\frac{1}{p}} \nonumber\\
&= \limsup_{n\to\infty} \|\lambda + w^{(n)}|\ms\|,\nonumber
\end{align}
where we used \eqref{pitt-l1b}.
Let
\begin{equation}\label{pitt-l2a}
\limsup_{n \to\infty} \| w^{(n)}|\ms\| = \lim_{i\rightarrow \infty} |\widetilde{Q}_{n_i}|^{\frac{1}{u}-\frac{1}{p}}\left(\sum_{k:Q_{0,k}\subset \widetilde{Q}_{n_i}} |w_k^{(n_i)}|^p \right)^{\frac{1}{p}}. 
\end{equation}
If the sequence of cubes $\widetilde{Q}_{n_i}$ satisfies the condition \eqref{pitt-l-alt1}, then $\|w^{(n_i)}|\ms\|\rightarrow 0$. So the inequality \eqref{pitt-l2} holds. 

If the sequence of cubes $\widetilde{Q}_{n_i}$ satisfies the condition \eqref{pitt-l-alt2}, then
\begin{align*}
  |\widetilde{Q}_{n_i}|^{\frac{1}{u}-\frac{1}{p}}&\left(\sum_{k:Q_{0,k}\subset \widetilde{Q}_{n_i}} \!\!| w_k^{(n_i)}|^p \right)^{\frac{1}{p}}  \\
  &\le \  
|\widetilde{Q}_{n_i}|^{\frac{1}{u}-\frac{1}{p}}\left(\sum_{k:Q_{0,k}\subset \widetilde{Q}_{n_i}} \!\!|\lambda_k+ w_k^{(n_i)}|^p \right)^{\frac{1}{p}} + \left(\frac{|\widetilde{Q}_{n_i}|}{|\widetilde{Q}(\lambda)|}\right)^{\frac{1}{u}-\frac{1}{p}}
\|\lambda|\ms\|.
\end{align*} 
But $(|\widetilde{Q}_{n_i}|/|\widetilde{Q}(\lambda)|)^{\frac{1}{u}-\frac{1}{p}}\rightarrow 0$, so  the inequality \eqref{pitt-l2} holds also in this case.

If the sequence of cubes $\widetilde{Q}_{n_i}$ satisfies the condition \eqref{pitt-l-alt3}, then for sufficiently large $i$ we have $\lambda|_{\widetilde{Q}_{n_i}}=0$, and again the inequality \eqref{pitt-l2} holds.\\

\emph{Substep 1.2.} Now we prove the inequality converse to \eqref{pitt-l2}, i.e., 
\begin{equation}\label{pitt-l3} 
 \limsup_{n\to\infty} \|\lambda + w^{(n)}|\ms\| \le  \max\big\{\|\lambda|\ms\|,\limsup_{n\to\infty} \|w^{(n)}|\ms\| \big\}\, .
\end{equation}
We can proceed in a similar way as in the last step, now using the cubes $Q_{n_i}$. If the sequence of cubes $\widetilde{Q}_{n_i}$ satisfies the condition \eqref{pitt-l-alt1}, then 
\[
|{Q}_{n_i}|^{\frac{1}{u}-\frac{1}{p}}\left(\sum_{k:Q_{0,k}\subset {Q}_{n_i}} \!\!| w_k^{(n_i)}|^p \right)^{\frac{1}{p}} \rightarrow 0 .
\]
So
\begin{equation}\label{pitt-l3a}
 \limsup_{n\to\infty} \|\lambda + w^{(n)}|\ms\| \le \|\lambda|\ms\| .
\end{equation}  
If the sequence of cubes $\widetilde{Q}_{n_i}$ satisfies the condition \eqref{pitt-l-alt2}, then similarly  as above we conclude
\begin{align*}
  |{Q}_{n_i}|^{\frac{1}{u}-\frac{1}{p}}&\left(\sum_{k:Q_{0,k}\subset {Q}_{n_i}} \!\!| \lambda_k + w_k^{(n_i)}|^p \right)^{\frac{1}{p}} \\
  &\le 
|{Q}_{n_i}|^{\frac{1}{u}-\frac{1}{p}}\left(\sum_{k:Q_{0,k}\subset {Q}_{n_i}} \!\!| w_k^{(n_i)}|^p \right)^{\frac{1}{p}}  + \left(\frac{|{Q}_{n_i}|}{|\widetilde{Q}(\lambda)|}\right)^{\frac{1}{u}-\frac{1}{p}}
\|\lambda|\ms\|
\nonumber 
\end{align*} 
 and the last summand tends to zero when $i\rightarrow \infty$. 
 
 If the sequence of cubes $\widetilde{Q}_{n_i}$ satisfies the condition \eqref{pitt-l-alt3}, then once more the sequences $\lambda + w_k^{(n_i)}$ and $ w_k^{(n_i)}$ coincide for sufficiently large $i$.\\

\emph{Step 2.} The general case is true by the density of finitely supported sequences in $\ms^{00}$ since the norm is a Lipschitzian function.  
\end{proof}

Now we can finally establish Pitt's theorem in our context.

\begin{theorem}\label{pitt-ms}
Let $1<p<u<\infty$ and $1\le q<\infty$. Then any bounded linear operator $T$ from $\ms^{00}(\Z^d)$ into $\ell_q(\Z^d)$ is compact.
\end{theorem}

\begin{proof}
Due to Proposition~\ref{predual-l2-2a}, Lemma \ref{predual-l3},  Proposition \ref{dual-p} and Lemma \ref{pitt-lemma} we can follow the arguments presented  in \cite{D}. We only sketch the proof for the convenience of the reader. 

The dual space to $\ms^{00}$ is separable, cf.  Proposition~\ref{predual-l2-2a}(ii) and  Proposition \ref{dual-p}, so every bounded sequence in $\ms^{00}$ has a weak Cauchy subsequence and $T$ is compact if it is weak-to-norm continuous.  

We may assume that $\|T\|=1$. Let $0<\varepsilon < 1$. We choose $x_\varepsilon\in \ms^{00}$ such that $\|x_\varepsilon|\ms\|=1$ and $1-\varepsilon\le \|T(x_\varepsilon)|\ell_q\|\le 1$. Let $w_n\rightharpoondown 0$ in $\ms^{00}$, and let 
$\|w^{(n)}|\ms\|\le M$. Lemma \ref{pitt-lemma} and the analogous statement for $\ell_q$, cf. \cite{D}, imply that
\begin{align*}\nonumber
  \|T(x_\varepsilon)|\ell_q\|^q + & t^q \limsup_{n\to\infty}\|T(w^{(n)})|\ell_q\|^q\ \\
  &=  \limsup_{n\to\infty} \|T(x_\varepsilon + w^{(n)})|\ell_q\|^q   \\
& \le   \limsup_{n\to\infty} \|x_\varepsilon + w^{(n)}|\ms\|^q \\
& =  \max \big(\|x_\varepsilon|\ms\|^q, t^q \limsup_{n\to\infty}\|(w^{(n)})|\ms\|^q \big),  \nonumber 
\end{align*}
where $t >0$. This leads to  
\[
\limsup_{n\to\infty}\|T(w^{(n)})|\ell_q\|^q \le t^{-q} \Big[\max\big(1,t^qM^q\big)-(1-\varepsilon)^q\Big] .
\]
The choice $0<\varepsilon\le \min (1, M^{-2q})$ and $t=\varepsilon^{\frac{1}{2q}}\ $ implies
\[
\limsup_{n\to\infty}\|T(w^{(n)})|\ell_q\|^q \le \varepsilon^{-1/2} \big(1-(1-\varepsilon)^q\big) .
\]
Taking the limit with $\varepsilon\rightarrow 0$ we have shown that  $\|T(w^{(n)})|\ell_q\|\rightarrow 0$. 
\end{proof}

\section{Finite dimensional Morrey sequence spaces}

Finally we shall briefly deal with finite dimensional sequence spaces related to $\ms$. We have at least two reasons for doing so: at first, in view of Theorem~\ref{cont}, there is never a compact embedding between two sequence spaces of Morrey type -- whereas any continuous embedding between finite-dimensional spaces is compact. Secondly, when dealing with smoothness Morrey spaces like $\MB$ or $\MF$, for instance, then wavelet decompositions usually lead to appropriate sequence spaces which should be studied in further detail. In this spirit it is quite natural and helpful to understand finite sequence spaces of Morrey type better than so far.

For the latter reason we do not consider finite Morrey sequence spaces as general as possible, but only a special `level' version of it.

\begin{definition}
Let $0<p\leq u<\infty$, $j\in\no$ be fixed and $\mathcal{K}_j = \{k: Q_{0,k}\subset Q_{-j,0}\}$ . We define
\begin{align}
\mmb = & \{ \lambda = \{\lambda_{k}\}_{k\in \mathcal{K}_j} \subset\C: \nonumber\\
&\quad  \|\lambda|\mmb\| = \sup_{Q_{-\nu,m}\subset Q_{-j,0}}\!\!
|Q_{-\nu,m}|^{\frac1u-\frac1p} \Big(\sum_{k:Q_{0,k}\subset Q_{-\nu,m}}\!\!|\lambda_k|^p\Big)^{\frac 1 p}<\infty\},
\end{align}
where the  supremum is taken over all $\nu\in\no$ and $m\in\Zn$ such that 
$Q_{-\nu,m}\subset Q_{-j,0}$.
\end{definition}

\begin{ownremark}
Similarly one can define  spaces related to any cube $Q_{-j,m}$, $m\in \Z^d$, but they are isometrically isomorphic to $\mmb$, so we restrict our attention to the last space. \\
Clearly, for $u=p$ this space coincides with the usual $2^{jd}$-dimensional space $\ell_p^{2^{jd}}$, that is, $m_{p,p}^{2^{jd}} = \ell_p^{2^{jd}}$. 
\end{ownremark}

\begin{lemma}\label{lemma15030}
Let $0< p_1\le u_1<\infty$,  $0< p_2\le u_2<\infty$, $j\in \no$ and $m_0\in\Zn$ be given. Then the norm of the compact identity operator  
\begin{equation}\label{id_j-m}
 \id_j: \mmbet\hookrightarrow \mmbzt
\end{equation}
satisfies
\begin{equation}\label{1503-0}
\|\id_j\| = 
\begin{cases}
1 &\qquad \text{if} \quad p_1\ge p_2\quad \text{and }\quad u_2\ge u_1,\quad  \\
 1 &\qquad \text{if} \quad p_1 < p_2\quad \text{and }\quad \frac{p_2}{u_2} \le \frac{p_1}{u_1},\\  
 2^{jd(\frac{1}{u_2}-\frac{1}{u_1})}&\qquad \text{if} \quad p_1\ge p_2\quad \text{and }\quad u_2 < u_1,\\
\end{cases}
\end{equation}
and in the remaining case, there is a constant $c$, $0<c\le 1$, independent of j such that 
\begin{equation}\label{1503-a}
c\, 2^{jd(\frac{1}{u_2}-\frac{p_1}{u_1p_2})} \le\|\id_j\| \leq 
2^{jd(\frac{1}{u_2}-\frac{p_1}{u_1p_2})}  \qquad \text{if} \quad  
p_1 < p_2\quad \text{and }\quad \frac{p_2}{u_2} > \frac{p_1}{u_1}\ .
\end{equation}
\end{lemma}

\begin{proof}
In case of $p_1\ge p_2$  the upper estimate for $\|\id_j\|$ follows from H\"older's inequality and the corresponding relations between $u_1$ and $u_2$. The lower estimate in case of $u_1\le u_2$ follows by applying the sequence $\lambda=\{\lambda_k\}_k$ with   $\lambda_{0}=1$ and $\lambda_{k}=0$ if $k\neq 0$ . Otherwise, if $u_1> u_2$, we can use the sequence $\lambda=\{\lambda_k\}_k$ with   $\lambda_k\equiv 1$ for any $k$. \\

Let now $p_1<p_2$ and $\frac{p_2}{u_2} \le \frac{p_1}{u_1}$.  If $\|\lambda|\mmbet\|=1$, then 
\[\sum_{k:Q_{0,k}\subset Q_{-\nu,m}}\!\!\! |\lambda_k|^{p_2} \le  \sum_{k:Q_{0,k}\subset Q_{-\nu,m}}\!\!\! |\lambda_k|^{p_1} ,\]
since $|\lambda_k|\le 1$. So for any $\nu$ with $0\le \nu\le j$ we have 
\begin{align*}
  2^{\nu d(\frac{p_2}{u_2}-1)} \sum_{k:Q_{0,k}\subset Q_{-\nu,m }} \!\!\! |\lambda_k|^{p_2} & \le  \;2^{\nu d(\frac{p_2}{u_2}-\frac{p_1}{u_1})}2^{\nu d(\frac{p_1}{u_1}-1)}  \sum_{k:Q_{0,k}\subset Q_{-\nu,m }}\!\!\! |\lambda_k|^{p_1} \\
  &\le  \ 2^{\nu d(\frac{p_2}{u_2}-\frac{p_1}{u_1})} \le 1.
\end{align*}
This proves that $\|\id_j\|\le 1$. The opposite inequality can be proved in the same way as in the first case. 

If   $p_1<p_2$,   $\frac{p_2}{u_2} > \frac{p_1}{u_1}$ and $\|\lambda|\mmbet\|=1$, then analogously as above we can prove that 
\begin{equation*}
2^{\nu d(\frac{p_2}{u_2}-1)} \sum_{k:Q_{0,k}\subset Q_{-\nu,m }} |\lambda_k|^{p_2} \le \ 2^{\nu d(\frac{p_2}{u_2}-\frac{p_1}{u_1})}. 
\end{equation*}
So  
\begin{equation*}
\|\lambda|\mmbzt\| \le 2^{dj(\frac{1}{u_2}-\frac{p_1}{u_1p_2})}.
\end{equation*} 
So  we obtain, 
\begin{equation*}
\|\lambda|\mmbzt\| \le 2^{dj(\frac{1}{u_2}-\frac{p_1}{u_1p_2})}. 
\end{equation*} 
To prove the opposite inequality we can use the same argument as in Substep 2.2. of the proof of Theorem \ref{cont}. We have
\begin{align} \nonumber
2^{dj(\frac{1}{u_2}-\frac{p_1}{u_1p_2})} & \le 2^{dj(\frac{1}{u_2}-\frac{1}{p_2})} 2^{dj(1-\frac{p_1}{u_1})\frac{1}{u_2}} \le c^{-1} 2^{dj(\frac{1}{u_2}-\frac{1}{p_2})} k_j \\ & = c^{-1}\|\lambda^{(j)}|\mmbzt\|\le \|\id_j\| \|\lambda^{(j)}|\mmbet\| \le c^{-1} \|\id_j\|, \nonumber
\end{align}
cf. \eqref{30.12-1}-\eqref{30.12-3}. 
\end{proof}

\begin{ownremark}
We suppose that \eqref{1503-a} is in fact an equality as well, but have no proof yet. In that case \eqref{1503-0} and \eqref{1503-a} could be summarised as
\[
\|\id_j\| \sim  2^{jd\left(\frac{1}{u_2}-\frac{1}{u_1} \min(1,\frac{p_1}{p_2})\right)_+}.
\]
\end{ownremark}

Finally we want to characterise the compactness of the embedding $\id_j$ given by \eqref{id_j-m} in some further detail. We restrict ourselves to the study of entropy numbers here, also for later use. Thus let us briefly recall the concept.

\begin{definition}\label{defi-ak}
Let $\ A_1\ $ and $\ A_2\ $ be two complex (quasi-) Banach spaces, 
$k\in\nat\ $ and let $\ T:A_1\rightarrow A_2\ $ be a linear and 
continuous operator from $\ A_1 $ into $\ A_2$. 
The {\em k\,th (dyadic) entropy number} $\ e_k\ $ of $\ T\ $ is the
infimum of all numbers $\ \varepsilon>0\ $ such that there exist $\ 2^{k-1}\ $ balls
in $\ A_2\ $ of radius $\ \varepsilon\ $ which cover the image $\ T\,U_1\ $ of the unit ball $\ U_1=\{a\in A_1 \;:\;\|a|A_1\|\leq 1\}$.
\end{definition}

For details and properties of entropy numbers we refer to \cite{CS,EE,Koe,Pie-s} (restricted to the case of Banach spaces), and \cite{ET} for some extensions to quasi-Banach spaces. Among other 
features we only want to mention 
the multiplicativity of entropy numbers: let $\ A_1$, $ A_2\ $ and $\ A_3$
be complex (quasi-) Banach spaces and $\ T_1 : A_1\longrightarrow A_2$, $ T_2 :
A_2\longrightarrow A_3\ $ two operators in the sense of 
Definition~\ref{defi-ak}. Then
\beq
e_{k_1+k_2-1} (T_2\circ T_1) \leq e_{k_1}(T_1)\,
e_{k_2} (T_2),\quad k_1, k_2\in\nat,
\label{e-multi}
\eeq

Note that $\ds\lim_{k\rightarrow\infty} e_k(T)= 0$ {if, and only if,} $T$ {is compact}, which explains the saying that entropy numbers measure
``how compact'' an operator acts.


One of the main tools in our arguments 
will be the characterisation of the asymptotic 
behaviour of 
the entropy numbers of the embedding $\ell_{p_1}^N 
\hookrightarrow \ell_{p_2}^N$. We recall it for convenience. For all $n \in \N$ we have in case of $0<p_1\leq p_2\leq\infty$ that 
\begin{equation}\label{Schuett}
e_k \Big(\id: \, \ell_{p_1}^N \hra \ell_{p_2}^N\Big) \sim
\left\{ \begin{array}{lll}
1 &\qquad & \mbox{if}\quad 1 \le k \le \log 2N\, , \\
\Big( \frac{\log (1+\frac{N}{k})}{k}\Big)^{\frac{1}{p_1}-\frac{1}{p_2}} 
&\qquad & \mbox{if}\quad \log 2N \le k \le 2N\, , \\
2^{-\frac{k}{2N}} \, N^{\frac{1}{p_2}-\frac{1}{p_1}}
&\qquad & \mbox{if}\quad 2N \le k\, , 
\end{array}\right.
\end{equation}
and in case $0<p_2<p_1\le \infty $ it holds 
\begin{equation}\label{Pietsch}  
e_k \Big(\id: \, \ell_{p_1}^N \hra \ell_{p_2}^N\Big) \sim 2^{-\frac{k}{2N}} \, N^{\frac{1}{p_2}-\frac{1}{p_1}} \qquad \text{for all}\quad k\in \N. 
\end{equation}
In the case $1 \le p_1,p_2 \le \infty$ this has been proved by
Sch\"utt \cite{Sch}. For $p_1<1$ and/or $p_2<1$  
we refer to Edmunds and Triebel \cite{ET} and
Triebel \cite[7.2,\, 7.3]{TrFS} (with a little supplement in \cite{Kue2}).


\begin{corollary}\label{lemma1503}
Let $j\in\nat$, $0<p_i\leq u_i<\infty$, $i=1,2$, and $k\in \no$ with $k\gtrsim 2^{jd}$. 
Then
\begin{equation}\label{1503-1}
e_k(\id_j:\mmbet \rightarrow \mmbzt) \sim  2^{-k2^{-jd}} \ 2^{jd\left(\frac{1}{u_2}-\frac{1}{u_1}\right)}.
\end{equation}
\end{corollary}

\begin{ownremark}
It will be obvious from the proof below that the assumption for $k$ to be sufficiently large, $k\gtrsim 2^{jd}$, is not needed in all cases. But for simplicity we have stated the result above in that setting only.
\end{ownremark}

\begin{proof} 
The estimate from above follows from the following commutative diagram and the multiplicativity of entropy numbers,
\[ 
\begin{CD}
\mmbet@>\id_j>> \mmbzt\\
@V{\id_1}VV @AA{\id_2}A\\
\ell^{2^{jd}}_{r_1} @>\id_{\ell,j}>> \ell^{2^{jd}}_{r_2}\, ,
\end{CD}
 \]
leads to
\begin{equation*}
e_k(\id_j) \leq \left\|\id_1:\mmbet\to\ell^{2^{jd}}_{r_1}\right\|\ \left\|\id_2: 
\ell^{2^{jd}}_{r_2}\to \mmbzt\right\| e_k\left(\id_{\ell,j}: \ell^{2^{jd}}_{r_1} \to
 \ell^{2^{jd}}_{r_2}\right).
\end{equation*}
We choose $r_1=p_1$ and $r_2=u_2$, apply Lemma~\ref{lemma15030} and arrive at
\[
e_k(\id_j) \leq 2^{jd\left(\frac{1}{p_1}-\frac{1}{u_1}\right)} \ e_k\left(\id_{\ell,j}: \ell^{2^{jd}}_{p_1} \to
 \ell^{2^{jd}}_{u_2}\right).
\]
Together with \eqref{Schuett} and \eqref{Pietsch} this leads to the upper estimate in \eqref{1503-1}, where only in case of $p_1\leq u_2$ the additional assumption $k\gtrsim 2^{jd}$ is needed.

Conversely, for the lower estimate we `reverse' the above diagram, that is, we consider 
\[ 
\begin{CD}
\mmbet@>\id_j>> \mmbzt\\
@A{\id_1}AA @VV{\id_2}V\\
\ell^{2^{jd}}_{r_1} @>\id_{\ell,j}>> \ell^{2^{jd}}_{r_2}\, .
\end{CD}
 \]
Thus we arrive at 
\begin{equation*}
e_k\left(\id_{\ell,j}: \ell^{2^{jd}}_{r_1} \to
 \ell^{2^{jd}}_{r_2}\right) 
\leq \left\|\id_1:\ell^{2^{jd}}_{r_1}\to\mmbet\right\|\ \left\|\id_2: 
\mmbzt\to\ell^{2^{jd}}_{r_2} \right\| e_k(\id_j) .
\end{equation*}
This time we choose $r_1=u_1$ and $r_2=p_2$, apply Lemma~\ref{lemma15030} again and obtain
\[
e_k\left(\id_{\ell,j}: \ell^{2^{jd}}_{u_1} \to
 \ell^{2^{jd}}_{p_2}\right)
 \leq 2^{jd\left(\frac{1}{p_2}-\frac{1}{u_2}\right)} \ e_k(\id_j).
\]
Together with \eqref{Schuett} and \eqref{Pietsch} this completes the argument of the lower estimate in \eqref{1503-1}, where now only in case of $u_1\leq p_2$ the additional assumption $k\gtrsim 2^{jd}$ is needed.
\end{proof}

\begin{ownremark}
It was not our aim here to study $e_k(\id_j)$ in all cases, though for several applications also the results (and constants) for small $k\in\nat$ are very useful. Moreover, there are further quantities which characterise compactness of operators which admit a lot of further interesting applications. At the moment we concentrated on the new concept of Morrey sequence spaces as introduced in this paper (with some forerunner in \cite{GuKiSch}) and found in this last section, that for sufficiently large $k\in\nat$, $k\gtrsim 2^{jd}$,
\[
e_k\left(\id_j:\mmbet \rightarrow \mmbzt\right) \sim e_k\left(\id: \ell_{u_1}^{2^{jd}} \to \ell_{u_2}^{2^{jd}}\right),
\]
though the corresponding sequence spaces are quite different.
\end{ownremark}

\end{document}